\documentclass[review,onefignum,onetabnum, final]{siamart171218}

\usepackage{amsfonts}
\usepackage{graphicx}
\usepackage{epstopdf}
\usepackage{algorithmic}
\usepackage{amsopn}
\usepackage{amssymb}

\usepackage{todonotes}

\usepackage{ulem}
\ifpdf
  \DeclareGraphicsExtensions{.eps,.pdf,.png,.jpg}
\else
  \DeclareGraphicsExtensions{.eps}
\fi

\newsiamremark{remark}{Remark}
\newsiamremark{hypothesis}{Hypothesis}
\crefname{hypothesis}{Hypothesis}{Hypotheses}
\newsiamthm{claim}{Claim}

\headers{Optimal Control of Perfect Plasticity}{C.~Meyer and S.~Walther}

\title{Optimal Control of Perfect Plasticity \\
Part II: Displacement Tracking\thanks{Submitted to the editors DATE.
\funding{This research was supported by the German Research Foundation (DFG) under grant 
number~ME~3281/9-1 within the priority program Non-smooth and Complementarity-based
Distributed Parameter Systems: Simulation and Hierarchical Optimization (SPP~1962).}}}

\author{
Christian Meyer\thanks{TU Dortmund, Faculty of Mathematics,
Vogelpothsweg 87, 44227 Dortmund, Germany 
  (\email{christian2.meyer@tu-dortmund.de},
  \url{http://www.mathematik.tu-dortmund.de/de/personen/person/Christian+Meyer.html},
  \email{stephan.walther@tu-dortmund.de},
  \url{http://www.mathematik.tu-dortmund.de/de/personen/person/Stephan+Walther.html}).}
\and Stephan Walther\footnotemark[2]}

\ifpdf
\hypersetup{
  pdftitle={Optimal Control of Perfect Plasticity \\Part II: Displacement Tracking},
  pdfauthor={C.~Meyer and S.~Walther}
}
\fi

\usepackage[notref,notcite]{showkeys}

\usepackage{esint}

\newcommand{\wordDefinition}[1]{\emph{#1}}
\newcommand{\shortspace}{\text{ \ \ }}
\newcommand{\mediumspace}{\text{ \ \ \ \ }}
\newcommand{\largespace}{\text{ \ \ \ \ \ \ }}

\newcommand{\bdot}{\boldsymbol{.}}
\newcommand{\boverdot}[1]{\overset{\bdot}{#1}}

\newcommand{\Rn}{\mathbb{R}^{n}}

\newcommand{\Rnns}{\mathbb{R}^{n \times n}_\textup{sym}}

\newcommand{\symnabla}{\nabla^{s}}

\newcommand{\lz}[1]{L^{2}(\Omega;#1)}

\newcommand{\scalarproduct}[3]{\left( #1 , #2 \right)_{#3}}

\newcommand{\norm}[2]{\|#1\|_{#2}}

\newcommand{\sequence}[2]{\{ #1_{#2} \}_{#2 \in \mathbb{N}}}

\renewcommand{\AA}{\mathcal{A}}

\newcommand{\EE}{\mathcal{E}}
\newcommand{\HH}{\mathcal{H}}

\newcommand{\UU}{\mathcal{U}}

\newcommand{\LL}{\mathcal{L}}

\newcommand{\KK}{\mathcal{K}}
\newcommand{\N}{\mathbb{N}}
\newcommand{\R}{\mathbb{R}}
\newcommand{\embed}{\hookrightarrow}
\newcommand{\weak}{\rightharpoonup}
\newcommand{\dual}[2]{\langle #1 , #2 \rangle}
\newcommand{\ddp}[2]{\frac{\partial #1}{\partial #2}}

\newcommand{\Rs}{\R^{d\times d}_\textup{sym}}
\newcommand{\Cb}{\mathbb{C}}

\newcommand{\Ab}{\mathbb{A}}
\renewcommand{\div}{\operatorname{div}}

\newcommand{\supp}{\operatorname{supp}}
\newcommand{\sign}{\operatorname{sign}}
\renewcommand{\div}{\operatorname{div}}
\renewcommand{\ker}{\operatorname{ker}}
\newcommand{\tr}{\operatorname{tr}}
\newcommand{\dist}{\operatorname{dist}}
\newcommand{\Mf}{\mathfrak{M}}
\newcommand{\BD}{\mathrm{BD}}
\renewcommand{\d}{\textup{d}}
\newcommand{\id}{\operatorname{Id}}
\newcommand{\e}{w}
\newcommand{\E}{W}
\newcommand{\bL}{\mathbf{L}}
\newcommand{\Lt}{\mathbb{L}}
\newcommand{\bW}{\mathbf{W}}
\newcommand{\Wt}{\mathbb{W}}
\newcommand{\bH}{\mathbf{H}}
\newcommand{\Ht}{\mathbb{H}}

\newtheorem{example}[theorem]{Example}
\newtheorem{notation and assumption}[theorem]{Notation and Assumption}

\begin{document}

\maketitle

\begin{abstract}
    The paper is concerned with an optimal control problem governed by the rate-independent system 
    of quasi-static perfect elasto-plasticity. 
    The objective is optimize the displacement field in the domain occupied by the body by means of prescribed Dirichlet boundary data, 
    which serve as control variables. The arising optimization problem is nonsmooth for several reasons, in particular, since the 
    control-to-state mapping is not single-valued. We therefore apply a Yosida regularization to obtain a single-valued 
    control-to-state operator. Beside the existence of optimal solutions, their approximation by means of this regularization approach
    is the main subject of this work.
    It turns out that a so-called reverse approximation guaranteeing the existence of a suitable recovery sequence can only be shown 
    under an additional smoothness assumption on at least one optimal solution.
\end{abstract}



\section{Introduction}
\label{sec:1}
In this paper, we investigate the following optimal control problem governed by the equations of 
\emph{quasi-static perfect plasticity} at small strain:
\begin{equation}\label{eq:optprobstrong}
\tag{P}
\left\{\quad
\begin{aligned}
	&\min \shortspace  J(u,u_D) := \Psi(u) + \frac{\alpha}{2}\norm{u_D}{H^1(0,T;H^2(\Omega;\R^n))}^2 \\
	&\;\begin{aligned}
	    &\text{s.t.}  \quad &
         - \div \sigma &= 0  &&\text{ in } \Omega, \\
        & & \sigma &= \mathbb{C} (\symnabla u - z)  &&\text{ in } \Omega, \\
	    & & \boverdot{z} &\in \partial I_{\mathcal{K}(\Omega)}(\sigma)   &&\text{ in } \Omega, \\
	    & & u &= u_D   &&\text{ on } \Gamma_D, \\
	    & & \sigma\nu &= 0   &&\text{ on } \Gamma_N, \\
	    & & u(0) &= u_0, \quad \sigma(0) = \sigma_0  &&\text{ in } \Omega,\\[1ex]
   &\text{and} \quad &  u_D(0) &= u_0 & & \text{ on }\Gamma_D.
    \end{aligned}
\end{aligned}\right.
\end{equation}
Herein, $u:(0,T)\times \Omega \to \R^n$, $n=2,3$, is the displacement field, while 
$\sigma, z:(0,T)\times \Omega \to \R^{n\times n}$ are the stress tensor and the plastic strain.
The boundary of $\Omega$ is split in two disjoint parts $\Gamma_D$ and $\Gamma_N$ 
with outward unit normal $\nu$. Moreover, $\Cb$ is the elasticity tensor 
and $\KK(\Omega)$ denotes the set of feasible stresses. The initial data $u_0$ and $\sigma_0$ 
are given and fixed. The Dirichlet data $u_D$ represent the control variable and
 $\alpha>0$ a fixed Tikhonov regularization parameter.
The objective $\Psi$ only contains the displacement field. Objectives involving the stress are considered in 
a companion paper \cite{meywal}.
This is the reason for calling \eqref{eq:optprobstrong} \emph{displacement tracking problem}. 
A mathematically rigorous version of \eqref{eq:optprobstrong} involving the function spaces 
and a rigorous notion of solutions for the state equation will be formulated in 
\cref{sec:stateeq} and \ref{sec:existence} below.
The precise assumptions on the data are given in \cref{sec:2}.
Regarding to a more detailed description of the plasticity model, we refer to \cite{OttosenRistinmaa2005:1}
and the references therein.

Some words concerning our choice of the control variable are in order:
In general, Dirichlet control problems provide particular difficulties due to regularity issues, when control functions 
in $L^2(\partial\Omega)$ are considered, see e.g.\ \cite{mayranvex13}. 
Nonetheless, we consider the Dirichlet displacement as control variables instead of distributed loads 
or forces on the Neumann boundary due to the \emph{safe load condition}. 
It is well known that the existence of solutions for the perfect plasticity system can only be shown under this additional condition
(see e.g.\ \cite{suquet, dalMaso}), 
which would lead to rather complex control constraints and it is a completely open question how to 
incorporate these constraints in the analysis of \eqref{eq:optprobstrong}. For this reason, 
we focus on the Dirichlet control problem. A possible realization of these controls 
by means of an additional linear elasticity equation avoiding the $H^2$-norm in the objective
 is elaborated in the companion paper \cite{meywal}.

Beside the safe-load condition, 
problem \eqref{eq:optprobstrong} exhibits several additional particular challenges. First of all, it is obviously nonsmooth 
due to the convex subdifferential appearing in the state equation. 
Moreover, the state equation is in general not uniquely solvable and its solutions significantly lack regularity, 
see \cite{suquet, dalMaso}.
Therefore, there is no single-valued control-to-state mapping and \eqref{eq:optprobstrong} should rather be regarded as 
an optimization problem in Banach space rather than an optimal control problem. 
Beside the existence of optimal solutions, our main goal is to approximate \eqref{eq:optprobstrong} via replacing 
$\partial I_{\KK(\Omega)}$ by its Yosida regularization. 
This is of course a classical procedure and, in order to show that the approximation works, i.e., that optimal solutions 
of the regularized problems converge to solutions of \eqref{eq:optprobstrong} (in a certain topology), 
the following steps have to be performed:
\begin{enumerate}
    \item The existence of (weak) accumulation points of sequences of optimal solutions of the regularized 
    problems have to be verified.
    \item Weak limits have to be feasible for the original problem \eqref{eq:optprobstrong}. 
    \item In order to show the optimality of the weak limit, one has to construct a \emph{recovery sequence} 
    for at least one optimal solution of the original problem.
\end{enumerate}
The last item is also known as \emph{reverse approximation} and might become a challenging task 
in the context of optimization of rate-independent systems, see \cite{MR09}.
This also happens to be the case here: 
In contrast to the perfect plasticity system, its regularized counterpart admits a unique solution with full regularity. 
It is therefore very unlikely that one can approximate \emph{every} solution of the perfect plasticity system 
by means of regularization and indeed, as classical examples demonstrate, this is in fact not true, see e.g.\ \cite{suquet} 
and \cref{ex:1d} below.
However, in the context of optimal control and optimization, respectively, we have the control as an additional variable at hand
and, in order to construct a recovery sequence, we have to find a sequence of \emph{tuples of state and control} feasible for 
the regularized problems so that the associated objective function values converge to the optimal value of \eqref{eq:optprobstrong}. This leads to much more flexibility in the construction of recovery sequences, provided 
that the set of controls is sufficiently rich. Unfortunately, this is not the case for our Dirichlet control and we need an 
additional control variable in terms of \emph{distributed loads} for the construction of a recovery sequence.
The idea is thus to introduce an additional load in the balance of momentum of the regularized problems and to drive
this load to zero for vanishing regularization parameter.
Our regularization procedure therefore does not only replace the convex subdifferential by its Yosida regularization, 
but also introduces a new additional control variable. To the best of our knowledge, this is a completely new idea.

Nevertheless, even with this additional control variable, we are only able to construct a recovery sequence under a fairly restrictive assumption. This assumption is caused by additional smoothness constraints as part of the regularized optimal control problems, 
which in turn are needed to pass to the limit in the regularized plasticity system, 
when the regularization parameter is driven to zero. If we assume that at least one optimal solution of the original 
(i.e., unregularized) optimization problem admits an admittedly high regularity, then we are able to construct 
a recovery sequence for this particular solution, which meets the smoothness constraints and is therefore
feasible for the regularized optimal control problems.
We thus obtain the desired approximation result under the assumption that there exists at least one ``smooth'' 
solution of \eqref{eq:optprobstrong}.

Let us put our work into perspective:
Quasi-static perfect plasticity is a rate-independent system. Optimization and optimal control of such systems 
have been considered by various authors and we only refer to \cite{Bro87, BK13, AO14, BK15, CHHM16, SWW17, Mun17, AC18, GW18} and the references therein.
Albeit still nonsmooth, optimization problems of this type substantially simplify, if the energy underlying the rate-independent system 
is uniformly convex. In quasi-static plasticity, this is the case, if hardening is present.
In this case, the plasticity system admits a unique solution in the energy space, which makes the construction of 
recovery sequences almost trivial. 
Nevertheless, the derivation of optimality conditions is still an intricate issue, see \cite{Wac12, Wac15, Wac16}. 
While all contributions mentioned so far deal with uniformly convex energies, 
the literature becomes rather scarce, when it comes to energies that lack strict convexity. 
In \cite{Rin08, Ste12, ELS13, EL14} the existence of optimal solutions for problems with non-convex energies are shown. 
To the best of our knowledge, the approximation of such problems has only been investigated in \cite{MR09, Rin09}, 
where a time-discretization instead of a regularization is considered. 
The approximation via discretization can however be hardly compared to our situation, since the discrete 
rate-independent systems are still not uniquely solvable so that there is still no (discrete) control-to-state map 
in contrast to the regularized setting. Therefore, the discrete optimization problems are still all but straight forward  
to solve, whereas the regularized optimal control problems are amenable for standard adjoint-based optimization methods.

The paper is organized as follows: 
After introducing our notation and standard assumptions in \cref{sec:2}, we introduce a rigorous notion of solution to the perfect plasticity 
system and recall the known results concerning the existence of solutions and the lack of uniqueness in \cref{sec:stateeq}. 
Then, \cref{sec:existence} is devoted to the existence of at least one (globally) optimal solution of \eqref{eq:optprobstrong}. 
In \cref{sec:reverse}, we lay the foundations for our reverse approximation argument for the construction of a recovery sequence, 
which is a basic ingredient for our main result in \cref{thm:main}. The last \cref{sec:convmin} covers this result and shows that 
solutions of \eqref{eq:optprobstrong} can indeed be approximated via Yosida regularization provided the mentioned 
regularity assumption is fulfilled.

\section{Notation and Standing Assumptions}
\label{sec:2}

We start with a short introduction in the notation used throughout the paper and in parallel 
list our standing assumptions. The latter are tacitly assumed for the rest of the
paper without mentioning them every time.

\paragraph{General notation}
Given two vector spaces $X$ and $Y$, we denote the space of linear and continuous functions
from $X$ into $Y$ by $\LL(X, Y)$. If $X = Y$, we simply write $\LL(X)$. 
The dual space of $X$ is denoted by $X^* = \LL(X, \mathbb{R})$.
If $H$ is a Hilbert space, we denote its scalarproduct by $\scalarproduct{\cdot}{\cdot}{H}$.
For the whole paper, we fix the final time $T > 0$.
To shorten the notation, Bochner-spaces are abbreviated by
$L^p(X) := L^p(0,T; X)$, $W^{1,p}(X) := W^{1,p}(0,T; X)$ ($p\in [1,\infty]$), and $C(X) := C([0,T]; X)$. 
Note that functions in $C(X)$ are continuous on the whole time interval.
When $G \in \LL(X; Y)$ is a linear and continuous operator, we can define an operator
in $\LL(L^p(X); L^p(Y))$ by $G(u)(t) := G(u(t))$ for all $u \in L^p(X)$ and for almost all $t \in [0,T]$,
we denote this operator also by $G$, that is, $G \in \LL(L^p(X); L^p(Y))$, and analog for
Bochner-Sobolev spaces, i.e., $G \in \LL(W^{1,p}(X); W^{1,p}(Y))$.

Given a coercive operator $G \in \LL(H)$ in a Hilbert space $H$, 
we denote its coercivity constant by $\gamma_G$, i.e., $\scalarproduct{Gh}{h}{H} \geq 
\gamma_G \norm{h}{H}^2$ for all $h \in H$.
With this operator we can define a new scalar product, which induces an equivalent norm, by
$H \times H \ni (h_1, h_2) \mapsto \scalarproduct{Gh_1}{h_2}{H} \in \mathbb{R}$.
We denote the Hilbert space equipped with this scalar product by $H_G$, that is
$\scalarproduct{h_1}{h_2}{H_G} = \scalarproduct{Gh_1}{h_2}{H}$ for all $h_1,h_2 \in H$.

If $p \in [1, \infty]$, then we denote its conjugate exponent by
$p'$, that is $\frac{1}{p} + \frac{1}{p'} = 1$.
Furthermore, $c,C>0$ are generic constants.

\paragraph{Matrices}

Given a matrix $\tau \in \R^{n\times n}$, we define its deviatoric (i.e., trace-free) part as 
\begin{equation*}
    \tau^D := \tau - \tfrac{1}{n} \tr(\tau) I 
\end{equation*}
and use the same notation for matrix-valued functions.
The Frobenius norm is denoted by $|A|_F^2 = \sum_{i,j=1}^n A_{ij}^2$ 
for $A\in \R^{n\times n}$ and for the associated scalar product, we write $A:B = \sum_{i,j=1}^n A_{ij} B_{ij}$, 
$A,B\in \R^{n\times n}$.
By $\Rnns$, we denote the space of symmetric matrices.

\paragraph{Domain}
The domain $\Omega\subset\R^n$, $n\in \mathbb{N}$, $n \geq 2$, is bounded 
of class $C^1$. The boundary consists of two disjoint 
measurable parts $\Gamma_N$ and $\Gamma_D$
such that $\Gamma=\Gamma_N \cup \Gamma_D$. While 
$\Gamma_N$ is a relatively open subset, $\Gamma_D$ is a relatively closed.
We moreover suppose that $\Gamma_D$ has a nonempty relative interior.
In addition, the set $\Omega \cup \Gamma_N$ is regular in the sense of Gr\"oger, cf.\ \cite{Gro89}.
Throughout the article, $\nu:\partial\Omega \to \R^n$ denotes the outward unit normal vector.

Thanks to the regularity of $\Omega$, 
the harmonic extension $\mathfrak{E}$ maps $C^1(\Gamma)$ to $W^{1,p}(\Omega)$ for some $p>n$.
Moreover, the maximum principle implies that 
\begin{equation}\label{eq:maxprin}
    \|\mathfrak{E}\varphi \|_{L^\infty(\Omega)} \leq \|\varphi\|_{L^\infty(\Gamma)}
    \quad \forall\, \varphi\in C^1(\Gamma).
\end{equation}

\begin{remark}
    The $C^1$-regularity of $\Omega$ and its boundary, respectively, is required for the trace theorem and the formula of 
    integration by parts for $\BD$-functions in \cite[Chap.~II, Theorem~2.1]{temam}, which will be used several times throughout the paper. 
    In \cite[Section~6]{francfort12}, it is claimed that this formula integration by parts also holds in Lipschitz domains, but no proof is provided.
    Since the minimal regularity of the boundary is not in the focus of this paper and would go beyond the scope of our work, 
    we restrict to domains of class $C^1$.
\end{remark}

\paragraph{Spaces}
Throughout the paper, by $L^p(\Omega; M)$ we denote Lebesgue spaces with values in $M$, 
where $p \in [1, \infty]$ and $M$ is a finite dimensional space.
To shorten notation, we abbreviate
\begin{equation*}
    \bL^p(\Omega) :=  L^p(\Omega; \R^n) \quad \text{and}
    \quad \Lt^p(\Omega) := L^p(\Omega;\Rnns).
\end{equation*}
Given $s\in \N$ and $p \in [1,\infty]$, the Sobolev spaces of vector- resp.\ tensor-valued functions are denoted by
\begin{equation*}
\begin{aligned}
    \bW^{s,p}(\Omega) &:= W^{s,p}(\Omega; \R^n), \quad & \bH^s(\Omega) &:= \bW^{s,2}(\Omega), \\
    \Wt^{s,p}(\Omega) &:= W^{s,p}(\Omega;\Rnns), \quad & \Ht^s(\Omega) &:= \Wt^{s,2}(\Omega).
\end{aligned}
\end{equation*}
Furthermore, set 
\begin{equation}\label{eq:sobolevdiri}    
    \bW_D^{1,p}(\Omega) := 
    \overline{\{\psi|_\Omega : \psi \in C^\infty_c(\R^n;\R^n),\; \supp(\psi) \cap \Gamma_D = \emptyset\}}^{\bW^{1,p}(\Omega)}
\end{equation}
and define $\bH^1_D(\Omega)$ analogously. The dual of $\bH^1_D(\Omega)$ is denoted by $\bH^{-1}_D(\Omega)$.
The space of \emph{bounded deformation} is abbreviated by 
\begin{equation*}
    \BD(\Omega) := \{ u \in \bL^1(\Omega) : 
    \tfrac{1}{2}(\partial_{i} u_j + \partial_{j} u_i) \in \Mf(\Omega) \;\forall\, i,j = 1, ..., n\},  
\end{equation*}
where $\Mf(\Omega)$ denotes the space of regular Borel measures on $\Omega$ and the (partial) derivatives are of course 
understood in a distributional sense. Equipped with the norm 
\begin{equation*}
    \|u\|_{\BD(\Omega)} := \|u\|_{\bL^1(\Omega)} + \sum_{i,j = 1}^n \tfrac{1}{2} \|\partial_{i} u_j + \partial_{j} u_i\|_{\Mf(\Omega)},
\end{equation*}
it becomes a Banach space.

\paragraph{Coefficients}
The elasticity tensor satisfies $\Cb \in \LL(\Rs)$ and is symmetric and coercive.
In addition we set $\mathbb{A} := \mathbb{C}^{-1}$ and note that $\Ab$ is symmetric and coercive, too.
Let us note that $\mathbb{C}$ could also depend on space, however, to
keep the discussion concise, we restrict ourselves to constant elasticity tensors.
 
\paragraph{Yield condition} 
The set defining the yield condition is denoted by $K \subset \Rnns$
and is closed and convex and there exists $0< \rho < R$ such that 
\begin{equation}\label{eq:Kinterior}
    \overline{B_{\R^{n\times n}}(0;\varrho)} \subset K \subset \overline{B_{\R^{n\times n}}(0;R)}.
\end{equation}
Given this set, we define the \emph{set of admissible stresses} as
\begin{equation*}
	\mathcal{K}(\Omega) := \{ \tau \in \Lt^2(\Omega) : \tau^D(x) \in K \text{ f.a.a. } x \in \Omega \}.
\end{equation*}
\begin{remark}
    The boundedness of the set $K$ is not really needed for our analysis. It is only required for the 
    formula of integration by parts in \eqref{eq:intbyparts}, which we only need to compare our 
    notion of solution to the one in \cite{dalMaso}. Nevertheless, we kept the boundedness assumption 
    on the set $K$, since it is fulfilled in all practically relevant examples such as e.g.\ the von Mises or 
    the Tresca yield condition.
\end{remark}
 
\paragraph{Operators}
Throughout the paper, $\nabla^s := \frac{1}{2}(\nabla + \nabla^\top): \bW^{1,p}(\Omega) \to \Lt^p(\Omega)$ 
denotes the linearized strain.
Its restriction to $\bW_D^{1,p}(\Omega)$ is denoted by the same symbol and, for the adjoint of this restriction, we write
$-\div := (\nabla^s)^* : \Lt^{p'}(\Omega) \to \bW^{1,p}_D(\Omega)^*$.

Let $\KK \subset \Lt^2(\Omega)$ be a closed and convex set. We denote the indicator function by
\begin{equation*}
I_{\KK} : \Lt^2(\Omega) \rightarrow \{ 0, \infty \}, \largespace \tau \mapsto 
\begin{cases}
	0, & \tau \in \KK, \\     
	\infty, & \tau \notin \KK.
\end{cases}                
\end{equation*}
By $\partial I_\KK : \Lt^2(\Omega) \rightarrow 2^{\Lt^2(\Omega)}$ we denote the subdifferential of the indicator function.
For $\lambda > 0$, the Yosida regularization is given by
\begin{equation}\label{eq:yosida}
	I_\lambda : \Lt^2(\Omega) \rightarrow \mathbb{R}, \largespace \tau \mapsto \frac{1}{2\lambda}
	\norm{\tau - \pi_\KK(\tau)}{\Lt^2(\Omega)}^2,
\end{equation}
where $\pi_\KK$ is the projection onto $\KK$ in $\Lt^2(\Omega)$, and its Fr\'echet derivative is
\begin{align*}
	\partial I_\lambda(\tau) = \frac{1}{\lambda} (\tau - \pi_\KK(\tau)).
\end{align*}
When $\lambda = 0$ we define $I_\lambda = I_0 := I_\KK$.
For a sequence $\sequence{\lambda}{n} \subset (0, \infty)$ we abbreviate $I_n := I_{\lambda_n}$.

\paragraph{Initial data}  For the initial stress field $\sigma_0$, we assume that 
$\sigma_0 \in \bW^{1,p}(\Omega)$ with some $p>n$.
Moreover, $\sigma_0$ satisfies the equilibrium condition, i.e., $\div \sigma_0 = 0$ a.e.\ in $\Omega$, 
and the yield condition, i.e., $\sigma_0 \in \KK(\Omega)$.
The initial displacement $u_0$ is supposed to be an element of $\bH^2(\Omega)$ and we require 
$\tr(\symnabla u_0 - \Ab \sigma_0) = 0$ a.e.\ in $\Omega$ in order to obtain a purely deviatoric initial 
plastic strain.

\begin{remark}
    The high regularity of $u_0$ is just needed to ensure that the feasible set of \eqref{eq:optprob} is nonempty. 
    For the mere discussion of the state system, this is not necessary.
    The same holds for the assumption $\sigma_0 \in \bW^{1,p}(\Omega)$, which will be needed to 
    construct a recovery sequence for the optimal control problem.
\end{remark}

\paragraph{Optimization Problem}
The Tikhonov parameter $\alpha$ is a positive constant and $\Psi$ is a functional that 
is bounded from below and satisfies a certain lower semicontinuity assumption w.r.t.\ weak convergence in the
displacement space, which will be made precise in \cref{sec:existence} below, see \eqref{eq:objlsc}.

\section{State Equation}\label{sec:stateeq}

We start our investigations with the analysis of the state system and recall some known results concerning 
quasi-static perfect plasticity. Already since the pioneering work of Suquet \cite{suquet},
it is well known that a precise definition of a solution to 
the system of perfect plasticity is all but straight forward, since a solution of the system in its ``natural'' form 
(below termed strong solution) does in general not exist due to a lack of regularity of the displacement and the 
plastic strain, respectively. 
We start with the definition of the function spaces already indicating this lack of regularity:

\begin{definition}[State spaces]
    \begin{enumerate}
        \item \emph{Stress space:}
        \begin{equation*}
            \Sigma(\Omega)  := \{ \tau  \in \Lt^2(\Omega) : \div \tau \in \bL^n(\Omega), \; \tau^D \in \Lt^\infty(\Omega)\}
        \end{equation*}         
        \item \emph{Displacement space:}
        \begin{equation*}
            \UU := \{u\in H^1(\bL^{\frac{n}{n-1}}(\Omega)) : \symnabla \boverdot u \in L^2_w(\Mf(\Omega;\Rnns))\}.
        \end{equation*}
        Herein, $L^2_w(\Mf(\Omega;\Rnns))$ is the space of weakly measurable functions with values in $\Mf(\Omega;\Rnns)$, 
        for which $t \mapsto \|\mu(t)\|_{\Mf}$ is an element of $L^2(0,T;\R)$. 
        For the definition of weak measurability, we refer to \cite[Section~8]{edwards}.
        
        We say that a sequence $\{u_n\} \subset \UU$ converges weakly in $\UU$ to $u$ and write $u_n \weak u$ in $\UU$, iff
        \begin{equation}\label{eq:defconvU}
           u_{n} \weak u \; \text{ in }H^1(\bL^{\frac{n}{n-1}}(\Omega)), \quad
            \symnabla \boverdot u_{n} \weak^* \symnabla \boverdot u \; \text{ in } L^2_w(\Mf(\Omega;\Rnns)).
        \end{equation}  
        Note that, by \cite[Theorem~8.20.3]{edwards}, $L^2_w(\Mf(\Omega;\Rnns)) = L^2(C_0(\Omega;\Rnns))^*$, 
        which gives a meaning to the weak-$\ast$ convergence in \eqref{eq:defconvU}.          
    \end{enumerate}
\end{definition}

\begin{remark}
    Unfortunately, $\BD(\Omega)$ does not admit the Radon-Nikod\'ym property and therefore weak measurability does not imply Bochner-measurability.
\end{remark}

\begin{definition}[Equilibrium condition]
	We define the set of stresses which fulfill the \wordDefinition{equilibrium condition} as
	\begin{align*}
		\mathcal{E}(\Omega) := \ker(\div) = \{ \tau \in \Lt^2(\Omega) : \scalarproduct{\tau}{\symnabla\varphi}{\Lt^2(\Omega)} = 0
		\;\; \forall\, \varphi \in \bH^1_D(\Omega) \}.
	\end{align*}
    Note that $\sigma \in \EE(\Omega)\cap \KK(\Omega)$ implies $\sigma \in \Sigma(\Omega)$.
\end{definition}

With the above definitions at hand, we can now define a hierarchy of three different solutions: 

\begin{definition}[Notions of solutions]\label{def:sol}
    Let $u_D\in H^1(\bH^1(\Omega))$ with $u_D(0) = u_0$ a.e.\ on $\Gamma_D$ be given. 
    Then we define the following notions of solutions to the perfect plasticity system:
    \begin{enumerate}
        \item \emph{Reduced solution:} A function $\sigma \in H^1(\Lt^2(\Omega))$ is called \emph{reduced solution} of the state equation, 
        if, for almost all $t\in (0,T)$, the following holds true:
        \vspace*{-4mm}
        \begin{subequations}
        \begin{gather}
           \intertext{$\bullet$ Equilibrium and yield condition:}
               \sigma(t) \in \EE(\Omega) \cap \KK(\Omega),
           \intertext{$\bullet$ Reduced flow rule inequality:}
               \int_\Omega \big(\mathbb{A} \boverdot{\sigma}(t) - \symnabla \boverdot{u}_D(t)\big) : \big(\tau - \sigma(t)\big) \,\d x \geq 0
        		\quad \forall \,\tau \in \mathcal{E}(\Omega) \cap \mathcal{K}(\Omega),\label{eq:flowrulered}
           \intertext{$\bullet$ Initial condition:}
        		\sigma(0) = \sigma_0.
        \end{gather}
        \end{subequations}
        \item \emph{Weak solution:} A tuple  $(u, \sigma) \in \UU \times H^1(\Lt^2(\Omega))$ is called \emph{weak solution} of the state equation, 
        if, for almost all $t\in (0,T)$, there holds
         \vspace*{-4mm}
        \begin{subequations}
        \begin{gather}
           \intertext{$\bullet$ Equilibrium and yield condition:}
               \sigma(t) \in \EE(\Omega) \cap \KK(\Omega), \label{eq:momyield}
           \intertext{$\bullet$ Weak flow rule inequality:}
           \begin{aligned}
               \int_\Omega \mathbb{A} \boverdot{\sigma}(t) : \big(\tau - \sigma(t)\big)\,\d x 
               + \int_\Omega \boverdot u(t)\cdot \div \big(\tau - \sigma(t)\big)\, \d x   \qquad\qquad  &\\
               \geq \int_\Omega \symnabla \boverdot u_D(t) : \big(\tau - \sigma(t)\big) 
               +  \boverdot u_D(t)\cdot \div \big(\tau - \sigma(t)\big)  \, \d x \qquad &\\[-1ex]
        		\forall \,\tau \in \Sigma(\Omega) \cap \mathcal{K}(\Omega), &
           \end{aligned}\label{eq:flowruleweak}
           \intertext{$\bullet$ Initial condition:}
        		u(0) = u_0, \quad \sigma(0) = \sigma_0. \label{eq:initial}
        \end{gather}
        \end{subequations}
       \item \emph{Strong solution:}
       A tuple $(u, \sigma) \in H^1(\bH^1(\Omega)) \times H^1(\Lt^2(\Omega))$ is called \emph{strong solution} of the state equation, 
       if, for almost all $t\in (0,T)$, there holds
         \vspace*{-4mm}
        \begin{subequations}
        \begin{gather}
           \intertext{$\bullet$ Equilibrium and yield condition:}
               \sigma(t) \in \EE(\Omega) \cap \KK(\Omega),
           \intertext{$\bullet$ Strong flow rule inequality:}
           \begin{aligned}
               \int_\Omega \mathbb{A} \boverdot{\sigma}(t) : \big(\tau - \sigma(t)\big)\,\d x 
               + \int_\Omega \symnabla \boverdot u(t) : \big(\tau - \sigma(t)\big) \, \d x  \geq 0 \qquad & \\[-1ex]
        		\forall \,\tau \in \mathcal{K}(\Omega), &
           \end{aligned}\label{eq:flowrulestrong}
           \intertext{$\bullet$ Dirichlet boundary condition:}
               u(t) - u_D(t) \in \bH^1_D(\Omega) \label{eq:diri}
           \intertext{$\bullet$ Initial condition:}
        		u(0) = u_0, \quad \sigma(0) = \sigma_0.
        \end{gather}
        \end{subequations}
    \end{enumerate}
\end{definition}

Some words concerning this definition are in order.
First, let us shortly investigate the relationship between the three different solution concepts. 
By restricting the test functions in \eqref{eq:flowruleweak} to functions in $\EE(\Omega)$, one immediately 
observes that every weak solution is also a reduced solution. 
Moreover, by integration by parts, 
it is evident that \eqref{eq:flowrulestrong} and \eqref{eq:diri} imply \eqref{eq:flowruleweak}. 
On the other hand, if a weak solution satisfies $u \in H^1(\bH^1(\Omega))$ and the Dirichlet boundary conditions in \eqref{eq:diri}, 
then integration by parts yields \eqref{eq:flowrulestrong}, provided that $\Sigma(\Omega)\cap \mathcal{K}(\Omega)$ 
is dense in $\KK(\Omega)$, which is a direct consequence of \cref{lem:density} proven in the appendix.
Thus, we have the following relations between the three different solution concepts:

\begin{corollary}[Relations between the solution concepts]\label{cor:solutions}
    \begin{enumerate}
        \item If $(u, \sigma)$ is a weak solution, then $\sigma$ is automatically a reduced solution.
        \item A weak solution $(u, \sigma)$ is a strong solution, 
        if and only if $u \in H^1(\bH^1(\Omega))$ and $(u - u_D)(t) \in \bH^1_D(\Omega)$ for all $t\in [0,T]$.
    \end{enumerate}
\end{corollary}

One may further ask why no Dirichlet boundary conditions appear in the definition of a weak solution. 
In fact, from a mechanical point of view, it is reasonable that no boundary conditions are imposed, 
since plastic slips may well develop on the Dirichlet part on the boundary, too, 
in form of tangential jumps of the displacement perpendicular to the outward normal $\nu$.
This observation is implicitly contained in the above definition as demonstrated in \cite[Theorem~6.1]{dalMaso}.
For convenience of the reader, we shortly sketch the underlying arguments.
To this end, suppose that a weak solution is given and let us define 
the \emph{plastic strain} $z \in L^2_w(\Mf(\Omega \cup \Gamma_D;\Rnns))$ by
\begin{equation}\label{eq:defp}
    z\lfloor_\Omega := \symnabla(u) - \Ab \sigma \, \d x,\quad 
    z\lfloor_{\Gamma_D} := (u - u_D) \odot \nu \HH^{n-1}, 
\end{equation}
where $\odot$ refers to the symmetrized dyadic product, i.e., $a\odot b = 1/2(a_i b_j + a_j b_i)_{i,j =1}^n$ 
for $a,b\in \R^n$.
Note that functions in $\BD(\Omega)$ admit traces in $L^1(\partial\Omega;\R^n)$ (see e.g.~\cite[Chap.~II, Thm.~2.1]{temam}) 
so that $z\lfloor_{\Gamma_D}$ is well defined.
According to \cite[Lemma~5.5]{dalMaso}, these equations carry over to the time derivatives for almost all $t\in (0,T)$, i.e.,
\begin{equation}\label{eq:pdot}
    \boverdot z\lfloor_\Omega := \nabla^s(\boverdot u) - \Ab \boverdot \sigma \,\d x,\quad 
    \boverdot z\lfloor_{\Gamma_D} := (\boverdot u - \boverdot u_D) \odot \nu \HH^{n-1}.
\end{equation}
Let us prove that the trace of $p$ vanishes. For this purpose, we need the following formula of integration by parts:

\begin{lemma}[{\cite[Chap.~II, Thm.~2.1]{temam}}]\label{lem:temam}
    For every $v\in \BD(\Omega)$ and every $\varphi\in W^{1,p}(\Omega)$, $p>n$, there holds
    \begin{equation}\label{eq:intbypartstemam}
        \int_\Omega \tfrac{1}{2}(v_i \partial_j \varphi + v_j \partial_i \varphi) \,\d x + \int_\Omega \varphi\, \d(\symnabla v)_{ij} 
        = \int_{\partial\Omega} \varphi \,\tfrac{1}{2}(v_i \nu_j + v_j \nu_i)\,\d s
    \end{equation}
    for  all $i,j = 1, ..., n$.
\end{lemma}

\begin{remark}
    The result in \cite{temam} is only stated for test functions in $C^1(\bar\Omega)$. However, the embeddings 
    $\BD(\Omega)\embed \bL^{\frac{n}{n-1}}(\Omega)$ and $W^{1,p}(\Omega) \embed C(\bar\Omega)$, $p>n$, 
    along with the trace theorem for $\BD$-functions and the density of $C^1(\bar\Omega)$ in $W^{1,p}(\Omega)$ 
    imply that the integration by parts also holds for test functions in $W^{1,p}(\Omega)$.
\end{remark}

Now, let $\varphi\in C^\infty_c(\Omega)$ be arbitrary. 
Then, since $\KK(\Omega)$ just acts on the deviatoric part, 
$\varphi\, \delta_{ij} + \sigma_{ij}(t) \in \Sigma(\Omega) \cap \KK(\Omega)$ for 
all $t\in [0,T]$ and therefore \eqref{eq:flowruleweak} and the above formula of integration by parts give
\begin{equation*}
    \sum_i \Big(\int_\Omega (\mathbb{A} \boverdot{\sigma})_{ii} \,\varphi \,\d x 
    + \int_\Omega \varphi\, \d (\symnabla\boverdot u)_{ii}\Big) = 0 \quad \forall\, \varphi \in C^\infty_c(\Omega)
\end{equation*}
and therefore $\tr\boverdot z\lfloor_\Omega = 0$ f.a.a.~$t\in (0,T)$.
Since $\tr (\symnabla(u_0) - \Ab \sigma_0) = 0$,
\cite[Theorem~7.1]{dalMaso} yields $\tr z\lfloor_\Omega = 0$ for all $t\in [0,T]$.
Similarly, we choose an arbitrary test function $\psi \in C^\infty(\Gamma)$ with $\supp(\psi) \subset \Gamma_D$ and 
test \eqref{eq:flowruleweak}  with $\mathfrak{E}\psi\, \delta_{ij} + \sigma_{ij}(t) \in \Sigma(\Omega) \cap \KK(\Omega)$. 
Note that $\mathfrak{E}\psi\, \delta_{ij} \in \Sigma(\Omega)$,  since the harmonic extension maps into $W^{1,p}(\Omega)$ 
with $p>n$.
Applying then again the formula of integration by parts implies, in view of 
$\tr\boverdot z\lfloor_\Omega = 0$, that 
\begin{equation}\label{eq:normaltrace}
    (\boverdot u - \boverdot u_D)\cdot \nu = 0\quad \text{a.e.~on } \Gamma_D.
\end{equation}
As $u_0 = u_D(0)$ a.e.\ on $\Gamma_D$, this yields $(u - u_D)\cdot \nu = 0$ a.e.~on $\Gamma_D$, giving in turn 
$\tr z\lfloor_{\Gamma_D} = 0$ for all $t\in [0,T]$.
Now that we know that $z$ is deviatoric, the formula of integration by parts from \cite[Proposition~2.2]{dalMaso}
is applicable, which yields
\begin{equation}\label{eq:intbyparts}
    \dual{\tau^D}{\boverdot z(t)} + \int_\Omega \tau : \big(\Ab \boverdot\sigma(t) - \symnabla(\boverdot u_D(t))\big)\,\d x 
    = \int_\Omega \div \tau \cdot (\boverdot u(t) - \boverdot u_D(t))\,\d x
\end{equation}
for almost all $t\in (0,T)$ and all $\tau \in \Sigma(\Omega)$.
It is to be noted that the duality product $ \dual{\tau^D}{\boverdot z}$ has to be treated with care, since, in general,
$\tau^D \notin C(\bar\Omega;\Rnns)$, but $\boverdot z$ is only a measure. For a detailed and rigorous discussion of this issue, 
we refer to \cite[Section~2.3]{dalMaso}.
Inserting \eqref{eq:intbyparts} in the flow rule inequality \eqref{eq:flowruleweak} then results in
\begin{equation}\label{eq:maxplastwork}
    \dual{\tau^D - \sigma^D(t)}{\boverdot z(t)} \geq 0
    \quad \forall\, \tau \in \Sigma(\Omega) \cap \KK(\Omega), 
\end{equation}
which is just the maximum plastic work inequality illustrating that $z$ as defined in \eqref{eq:defp} is indeed 
the correct object for the plastic strain.
As a byproduct, we obtain the second equation in \eqref{eq:defp} as boundary condition on $\Gamma_D$
indicating that the Dirichlet boundary condition in \eqref{eq:diri} as part of the definition of a strong solution 
is in general too restrictive as already mentioned above. 
Accordingly, a strong solution does in general not exist, while we have the following result for a weak solution: 

\begin{proposition}[Existence of weak solutions, {\cite[R\'esultat~2]{suquet}}]
    For all $u_D \in H^1(\bH^1(\Omega))$, there exists a weak solution in the sense of \cref{def:sol}.
\end{proposition}

\begin{proof}
   Using the Yosida regularization,  Suquet showed in \cite{suquet} the existence of a functions 
   $\sigma \in H^1(\Lt^2(\Omega))$ and $v\in L^2_w(\BD(\Omega))$ so that, for almost all $t\in (0,T)$,
    \begin{equation}\label{eq:suquetsys}
    \begin{aligned}
        -\div \sigma(t) &\in \EE(\Omega) \cap \KK(\Omega),\\
        \int_\Omega \mathbb{A} \boverdot{\sigma}(t) : & \,\big(\tau - \sigma(t)\big)\,\d x 
        + \int_\Omega v(t)\cdot \div \big(\tau - \sigma(t)\big)\, \d x  \\
        &\geq \dual{\boverdot u_D(t)}{(\tau - \sigma(t))\nu}_{H^{1/2}(\Gamma_D), H^{-1/2}(\Gamma_D)}
        \quad \forall \,\tau \in \Sigma(\Omega) \cap \mathcal{K}(\Omega), \\
        \sigma(0) &= \sigma_0. 
    \end{aligned}              
    \end{equation}
    Due to the continuous embedding $\BD(\Omega) \embed \bL^{\frac{n}{n-1}}(\Omega)$ 
    (see e.g.~\cite[Chap.~II, Theorem~2.2]{temam}) and the Radon-Nikodym property of $\bL^{\frac{n}{n-1}}(\Omega)$, we have that 
    $v\in L^2(\bL^{\frac{n}{n-1}}(\Omega))$. Therefore, 
    \begin{equation*}
        u(t) := u_0 + \int_0^t v(r)\,\d r
    \end{equation*}
    is an element of $H^1(\bL^{\frac{n}{n-1}}(\Omega))$ and satisfies the initial condition in \eqref{eq:initial}.
    Inserting this in \eqref{eq:suquetsys} and integrating the right hand side by parts gives the desired flow 
    rule inequality \eqref{eq:flowruleweak}. The claimed regularity of $u$ directly follows from the regularity of $v = \boverdot u$.
\end{proof}

\begin{remark}[Other equivalent notions of solutions]
    Beside the reformulation of the flow rule in terms of the maximum plastic work inequality \eqref{eq:maxplastwork}, 
    there are other solutions concepts, which are equivalent to the definition of a weak solution, 
    such as the notion of a \emph{quasi-static evolution}, which in essence corresponds to a global energetic solution 
    in the sense of \cite{MielkeRoubicek2015}. 
    For an overview over the various notions of solutions and a rigorous proof of their equivalence, we refer to 
    \cite[Section~6]{dalMaso}. A slightly sloppy, but very illustrating derivation of the flow rule 
    out of the quasi-static evolution can also be found in \cite{francfort16}.    
\end{remark}

Unfortunately, the weak solution is not unique, as the following example shows:

\begin{example}[{\cite[Section~2.1]{suquet}}]\label{ex:1d}
    We choose $\Omega = (0, 1)$, $\Gamma_D = \partial\Omega = \{ 0,1 \}$, $T = 1$, 
    $K = [-1, 1]$, $\mathbb{C} = 1$, $(\sigma_0, u_0) = 0$, and $u_D(t,x) := 2tx$.
One easily verifies that the stress does only depend on the time with
$\sigma(t) = 2t$ for $t \in (0, \frac{1}{2})$ and $\sigma(t) = 1$ for $t \in (\frac{1}{2},1)$.
For the displacement one obtains $u(t,x) = 2tx$ for $(t,x) \in (0,\frac{1}{2}) \times (0,1)$
so that it is unique for $t \in (0,\frac{1}{2})$. For $t \in (\frac{1}{2}, 0)$ there are more than one solution, for example
\begin{align*}
	u(t,x) &= 2tx, \text{\, if \,} (t,x) \in (\tfrac{1}{2},1) \times (0, 1), \\
	u(t,x) &=
	\left\{
	\begin{array}{l l}
	\frac{2tx}{\beta} + x - \frac{x}{\beta}, &\text{\, if \,} (t,x) \in (\frac{1}{2},1) \times [0, \beta],  \\[0.5ex] 
	2t + x - 1,  &\text{\, if \,} (t,x) \in (\frac{1}{2},1) \times [\beta, 1],
	\end{array}
	\right.\\
	u(t,x) &=
	\left\{
	\begin{array}{l l}
	x, &\text{\, if \,} (t,x) \in (\frac{1}{2},1) \times [0, \beta],  \\  [0.5ex] 
	\alpha t + x - \frac{\alpha}{2},  &\text{\, if \,} (t,x) \in (\frac{1}{2},1) \times [\beta, 1],
	\end{array}
	\right.
\end{align*}
where $\alpha\in [0,2]$ and $\beta\in [0,1]$ can be freely chosen.
Note that the last solution just provides the minimal regularity, i.e., $\partial_x \boverdot u(t) \in \Mf(\Omega)$.
\end{example}

The uniqueness of the stress field observed in the above example is a general result:

\begin{lemma}[Uniqueness of the stress, {\cite[Theorem~1]{johnson}, \cite[Lemma~3.5]{meywal}}]
\label{lem:stressunique}
	Assume that $\sigma_1, \sigma_2 \in H^1(\Lt^2(\Omega))$ are	two reduced solutions.	Then $\sigma_1 = \sigma_2$.
\end{lemma}

\begin{remark}[Optimal control vs.\ optimization]
    Since the displacement field as part of a weak solution is not unique in general, there is no 
    (single-valued) control-to-state operator mapping $u_D$ to $u$.
    Therefore, one might argue that \eqref{eq:optprobstrong} is actually no real optimal control problem. 
    Strictly speaking, one should rather regard it as an \emph{optimization problem} with the triple $(u, \sigma, u_D)$ 
    as optimization variables.
\end{remark}

\section{Existence of Optimal Solutions}\label{sec:existence}

Before we come to the main point of our analysis, which concerns the approximation of \eqref{eq:optprobstrong} 
by means of regularized optimal control problems, let us address the existence of optimal solutions.
The proof in principle follows the classical direct method, for which we need the following boundedness and 
continuity results:

\begin{lemma}[{\cite[Lemma~3.6]{meywal}}]\label{lem:sigmabound}
    Let $u_D\in H^1(\bH^1(\Omega))$ be given and $\sigma$ be the associated reduced solution. Then there holds
    \begin{equation}\label{eq:sigmadotbound}
        \|\boverdot\sigma\|_{L^2(\Lt^2(\Omega))} \leq \gamma_{\Ab}^{-1}\,\|u_D\|_{H^1(\bH^1(\Omega))},
    \end{equation}        
    where $\gamma_{\Ab}$ is the coercivity constant of $\Ab$.
    Consequently, there is a constant $C>0$ such that 
    $\|\sigma\|_{H^1(\Lt^2(\Omega))} \leq C (\|\sigma_0\|_{\Lt^2(\Omega)} + \|u_D\|_{H^1(\bH^1(\Omega))})$.
\end{lemma}

\begin{lemma}[Continuity of reduced solutions, {\cite[Proposition~3.10]{meywal}}]\label{lem:contred}
   	Let $\{ u_{D,n} \} \subset H^1(\bH^1(\Omega))$ be a sequence such that
	\begin{equation}\label{eq:convdiri}
	\begin{gathered}
		u_{D,n} \rightharpoonup u_D \;\text{ in } H^1(\bH^1(\Omega)), \quad
		u_{D,n} \rightarrow u_D \;\text{ in } L^2(\bH^1(\Omega)), \\
		u_{D,n}(T) \rightarrow u_D(T) \;\text{ in } \bH^1(\Omega)		
	\end{gathered}
	\end{equation}
    and denote the (unique) reduced solution associated with $u_{D,n}$ by $\sigma_n$.
    Then $\sigma_n \rightharpoonup \sigma$ in $H^1(\Lt^2(\Omega))$, where $\sigma$ is the reduced solution w.r.t.~$u_D$.
\end{lemma}

\begin{lemma}\label{lem:ubound}
    There is a constant $C>0$, independent of $u_D$, such that every weak solution 
    w.r.t.~$u_D$ fulfills 
    \begin{equation*}
        \Big(\int_0^T\|\boverdot u(t)\|^2_{\BD(\Omega)}\,\d t\Big)^{1/2} \leq C\,\|u_D\|_{H^1(\bH^1(\Omega))}  \big(1 + \|u_D\|_{H^1(\bH^1(\Omega))}\big).
    \end{equation*}
\end{lemma}

\begin{proof}
    Let $\varphi \in C^\infty_c(\Omega)$ with $\|\varphi\|_{L^\infty(\Omega)} \leq 1$ and $i,j \in \{1, ..., n\}$ be arbitrary.
    According to \eqref{eq:Kinterior}, the test function 
    \begin{equation*}
        (\tau_\varphi)_{ij}  = (\tau_\varphi)_{ji} := - \frac{\varrho}{\sqrt{2}}\,\varphi, \quad 
        (\tau_\varphi)_{kl} = 0 \quad \forall\, (k,l) \notin  \{(i,j), (j,i)\}
    \end{equation*}
    is admissible for \eqref{eq:flowruleweak}. Using $\div \sigma = 0$, we deduce
    \begin{equation*}
        \int_\Omega \varphi \,\d (\symnabla \boverdot u)_{ij} 
        \leq  \frac{\sqrt{2}}{\varrho} \Big(\int_\Omega \symnabla \boverdot u_D : \sigma \,\d x 
        - \int_\Omega \Ab \boverdot \sigma : (\tau_\varphi - \sigma)\,\d x\Big)
    \end{equation*}        
    and consequently, since $\varphi \in C^\infty_c(\Omega)$ with $\|\varphi\|_{L^\infty(\Omega)} \leq 1$
    was arbitrary,     
    \begin{equation*}
    \begin{aligned}
        \|\symnabla \boverdot u\|_{L^2_w(\Mf(\Omega;\Rnns))}
        &\leq C\big(
        \begin{aligned}[t]
            & \|u_D\|_{H^1(\bH^1(\Omega))}\,\|\sigma\|_{L^\infty(\Lt^2(\Omega))} \\
            & + \|\boverdot \sigma\|_{L^2(\Lt^2(\Omega))} 
            + \|\boverdot \sigma\|_{L^2(\Lt^2(\Omega))} \|\sigma\|_{L^\infty(\Lt^2(\Omega))}\big)        
        \end{aligned}\\
        & \leq C\, \|u_D\|_{H^1(\bH^1(\Omega))} \big(1 + \|u_D\|_{H^1(\bH^1(\Omega))}\big),
    \end{aligned}
    \end{equation*}
    where we used \cref{lem:sigmabound}.
    
    Since $\Gamma_D$ is assumed to have a nonempty relative interior, there is a set $\Lambda \subset \Gamma_D$
    and a constant $\delta>0$ such that $\Lambda$ has positive boundary measure and $\dist(\Lambda, \partial\Gamma_D) \geq \delta$. 
    By \cite[Chap.~II, Theorem~2.1]{temam}, $\boverdot u(t)$ admits a trace in $\bL^1(\Gamma)$ for almost all $t\in (0,T)$. 
    In the following, we neglect the variable $t$ for the sake of readability.
    The restriction of  this trace to $\Lambda$ is denoted by $\boverdot u|_{\Lambda}$. 
    We extend $\sign(\boverdot u|_\Lambda)$ (where the sign is to be understood componentwise) to the whole boundary $\Gamma$ 
    by zero and apply convolution with a smoothing kernel to obtain a sequence of functions $\{\varphi_n\}\subset C^\infty(\Gamma;\R^n)$ 
    with $\supp(\varphi_n) \subset \Gamma_D$ (thanks to $\dist(\Lambda, \partial\Gamma_D) \geq \delta$)
    and $\|\varphi_n\|_{L^\infty(\Gamma;\R^n)}\leq 1$ for all $n\in \N$.
    Given these functions, let us define
    \begin{equation*}
        (\tau_n)_{ij} = \frac{\varrho}{\sqrt{2}}\, \mathfrak{E}(\varphi_{n,i} \,\nu_j + \varphi_{n,j} \,\nu_i),
    \end{equation*}
    where $\mathfrak{E}$ denoted the harmonic extension and $\nu$ is the outward normal.  
    Then, \eqref{eq:maxprin} implies $\|\tau_n\|_{\Lt^\infty(\Omega)} \leq \varrho$ and, since in addition 
    $\tau_n$ vanishes on $\Gamma_N$ by construction, we have  $\tau_n\in \Sigma(\Omega) \cap \KK(\Omega)$. 
    Note that, by the mapping properties of $\mathfrak{E}$, $\tau_n\in \Wt^{1,p}(\Omega) \embed \Sigma(\Omega)$.
    If we insert this as test function in \eqref{eq:flowruleweak} and apply again the integration by parts from \cref{lem:temam},  
    then $\div\sigma = 0$ and \eqref{eq:normaltrace} imply
    \begin{equation*}
        \int_{\Gamma_D} \varphi_n \cdot \boverdot u\, \d s
        \leq \frac{\sqrt{2}}{\varrho}\Big(\int_\Omega \tau_n : \d \symnabla (\boverdot u) 
        - \int_\Omega \symnabla \boverdot u_D : \sigma \,\d x + \int_\Omega \Ab \boverdot \sigma : (\tau_n - \sigma)\,\d x\Big).
    \end{equation*}
    Now, since $\varphi_n \to \sign(\boverdot u)$ a.e.\ in $\Lambda$,  $\varphi_n \to 0$ a.e.\ in $\Gamma_D\setminus\Lambda$
    and $|\varphi_n \cdot \boverdot u| \leq |\boverdot u|$ a.e.\ on $\Gamma_D$, Lebesgue's dominated convergence theorem along 
    with our previous estimate gives
    \begin{equation*}
    \begin{aligned}
        \|\boverdot u\|_{L^2(\bL^1(\Lambda))} \leq C\, \|u_D\|_{H^1(\bH^1(\Omega))} \big(1 + \|u_D\|_{H^1(\bH^1(\Omega))}\big).
    \end{aligned}
    \end{equation*}
    Thanks to \cite[Chap.~II, Proposition~2.4]{temam}, this completes the proof.
\end{proof}

\begin{remark}
    A priori estimates for quasistatic evolutions (which is an equivalent notion of solution as mentioned above) 
    are already proven in \cite[Thm.~5.2]{dalMaso} in a slightly different setting.
\end{remark}

\begin{lemma}\label{lem:convU}
    Let $\{u_n\}\subset \UU$ be a sequence such that, for all $n\in \N$, 
    \begin{equation}\label{eq:boundU}
        u_n(0) = u_0 \quad \text{and} \quad \int_0^T \|\boverdot u_n(t)\|_{\BD(\Omega)}^2 \,\d t \leq C
    \end{equation}
    with a constant $C>0$. Then there exists a subsequence converging weakly in $\UU$ as defined in \eqref{eq:defconvU}.
\end{lemma}

\begin{proof}
    Owing to \eqref{eq:boundU}, $\{\symnabla \boverdot u_n\}$ is bounded in $L^2_w(\Mf(\Omega;\Rnns))$, which, 
    according to \cite[Theorem~8.20.3]{edwards}, is the dual of $L^2(C_0(\Omega;\Rnns))^*$. 
    Thus, there exists a subsequence such that 
    \begin{equation}\label{eq:convstrain}
        \symnabla \boverdot u_{n_k} \weak^* \e \quad \text{in } L^2_w(\Mf(\Omega;\Rnns)).
    \end{equation}
    Due to $\BD(\Omega) \embed \bL^{\frac{n}{n-1}}(\Omega)$, $\{\boverdot u_{n_k}\}$ is bounded in 
    $L^2(\bL^{\frac{n}{n-1}}(\Omega))$ and, since all $u_n$ share the same initial value, $\{u_{n_k}\}$ is bounded in 
    $H^1(\bL^{\frac{n}{n-1}}(\Omega))$ so that, by reflexivity, there is another subsequence 
    (denoted w.l.o.g.\ by the same symbol) such that 
    \begin{equation}\label{eq:convdisp}
        u_{n_k} \weak u \quad \text{in } H^1(\bL^{\frac{n}{n-1}}(\Omega)).
    \end{equation}
    Now, for every $\tau \in C^\infty_c(\Omega;\Rnns)$ and every $\varphi\in C^\infty_c(0,T)$,
    \eqref{eq:convstrain} and \eqref{eq:convdisp} imply
    \begin{equation*}
    \begin{aligned}
        \int_0^T \dual{\e(t)}{\tau} \varphi(t)\,\d t
        &= \lim_{k\to\infty} \int_0^T \dual{\symnabla \boverdot u_{n_k}(t)}{\tau} \varphi(t)\,\d t \\
        &= \lim_{k \to\infty} \int_0^T \int_\Omega \boverdot u_{n_k}(t) \cdot \div \tau\,\d x\, \varphi(t)\,\d t  \\
        & = \int_0^T \int_\Omega \boverdot u(t) \cdot \div \tau\,\d x\, \varphi(t)\,\d t  
    \end{aligned}
    \end{equation*}
    and hence $\e(t) = \symnabla \boverdot u(t)$ a.e.~in $(0,T)$. 
\end{proof}

\begin{proposition}[Continuity properties of weak solutions]\label{prop:contweak}
	Let $\{ u_{D,n} \}_{n \in \mathbb{N}} \subset H^1(\bH^1(\Omega))$ be a sequence fulfilling \eqref{eq:convdiri}.
    Then, there is a subsequence of weak solutions $\{u_{n_k}, \sigma_{n_k}\}_{k \in \N}$ associated with
    $\{u_{D,n_k}\}$ such that 
    \begin{gather*}
        \sigma_{n_k} \weak \sigma \;\text{ in } H^1(\Lt^2(\Omega)), \quad 
        u_{n_k} \weak u \; \text{ in }\UU,
    \end{gather*}    	
	and the weak limit $(u,\sigma)$ is a weak solution associated with the limit $u_D$.
\end{proposition}

\begin{proof}
    Since we already know that the stress component of every weak solution is also a reduced one and the latter is unique 
    by \cref{lem:stressunique}, the convergence of the stresses follows from \cref{lem:contred} (even for the whole sequence).
    
    Owing to \cref{lem:ubound}, $\{\boverdot u_{n}\}$ fulfills the boundedness assumption in \eqref{eq:boundU} so that, 
    by \cref{lem:convU}, there is a subsequence $\{u_{n_k}\}$ converging weakly in $\UU$ to some limit $u\in \UU$.
    Due to $H^1(\bL^1(\Omega))\embed C(\bL^1(\Omega))$, the weak limit $u$ also satisfies the initial condition.
	
	It remains to prove that $(u, \sigma)$ fulfills the flow rule inequality \eqref{eq:flowruleweak}.
    To this end, choose an arbitrary $\tau \in L^2(\Lt^2(\Omega))$ with $\tau(t) \in \Sigma(\Omega) \cap \mathcal{K}(\Omega)$	for almost all $t \in [0,T]$. 
    Then, the flow rule inequality for $(u_{n_k}, \sigma_{n_k})$ along with $\div \sigma_{n_k} = 0$
    and the (weak) convergences of $u_{D, n_k}$, $u_{n_k}$, and $\sigma_{n_k}$ yields
    \begin{equation}\label{eq:flowlimit}
	\begin{aligned}
	    & \liminf_{k\to \infty} \scalarproduct{\mathbb{A}\boverdot{\sigma}_{n_k}}{\sigma_{n_k}}{L^2(\Lt^2(\Omega))} \\ 
	    & \leq \lim_{k\to\infty} 
        \begin{aligned}[t]
    	    \Big[ & \scalarproduct{\mathbb{A}\boverdot{\sigma}_{n_k} - \symnabla \boverdot{u}_{D,n_k}}{\tau}{L^2(\Lt^2(\Omega))} \\
 	        & + \int_0^T \int_\Omega(\boverdot{u}_{n_k} - \boverdot{u}_{D,n_k})\div\tau\,\d x \d t
		     -  \scalarproduct{\symnabla \boverdot{u}_{D,n_k}}{\sigma_{n_k}}{L^2(\Lt^2(\Omega))} \Big]
        \end{aligned} \\
		 & = \scalarproduct{\mathbb{A}\boverdot{\sigma} - \symnabla \boverdot{u}_{D}}{\tau}{L^2(\Lt^2(\Omega))} 
	    + \int_0^T \int_\Omega(\boverdot u -  \boverdot{u}_{D})\div\tau\,\d x \d t
		 -  \scalarproduct{\symnabla \boverdot{u}_{D}}{\sigma}{L^2(\Lt^2(\Omega))}
	\end{aligned}    
    \end{equation}
    where we used Lemma~3.9 in our companion paper~\cite{meywal} for the convergence of the last term.
	On the other hand, the weak lower semicontinuity of $\norm{\cdot}{\Lt^2(\Omega)_\mathbb{A}}$ together with
	$H^1(\Lt^2(\Omega)) \hookrightarrow C(\Lt^2(\Omega))$ gives
	\begin{equation}\label{eq:flowlimit2}
    \begin{aligned}
		& \liminf_{k \rightarrow \infty}
		\scalarproduct{\mathbb{A} \boverdot{\sigma}_{n_k}}{\sigma_{n_k}}{L^2(\Lt^2(\Omega))}\\
		&\qquad = \frac{1}{2} \liminf_{k \rightarrow \infty} \norm{\sigma_{n_k}(T)}{\Lt^2(\Omega)_\Ab}^2
		- \frac{1}{2} \norm{\sigma_0}{\Lt^2(\Omega)_\Ab}^2\\
		&\qquad \geq \frac{1}{2} \norm{\sigma(T)}{\Lt^2(\Omega)_\Ab}^2
		- \frac{1}{2} \norm{\sigma_0}{\Lt^2(\Omega)_\Ab}^2 
		= \scalarproduct{\mathbb{A} \boverdot{\sigma}}{\sigma}{L^2(\Lt^2(\Omega))}.    
    \end{aligned}
	\end{equation}
	Together with \eqref{eq:flowlimit} and $\div \sigma = 0$, this implies the flow rule inequality for the weak limit.
\end{proof}

Given these boundedness and continuity results, we can now establish the existence of at least one optimal solution. 
Before we do so, let us recall our optimization problem and state it in a rigorous manner:
\begin{equation}\tag{P}\label{eq:optprob}
\left\{\,
\begin{aligned}
	\min \shortspace & J(u, u_D) := \Psi(u) + \frac{\alpha}{2}\,\|u_D\|_{H^1(\bH^2(\Omega))}^2 \\
	\text{s.t.} \shortspace & u_D \in H^1(\bH^{2}(\Omega)), \quad (u,\sigma) \in \UU \times H^1(\Lt^2(\Omega)),\\
	&  (u,\sigma) \text{ is a weak solution w.r.t.\ } u_D,
	\quad \text{and} \quad u_D(0) - u_0 \in \bH^1_D(\Omega),
\end{aligned}\right.
\end{equation}
where $\Psi : \UU \to \R$ is bounded from below and 
lower semicontinous w.r.t.~weak convergence in $\UU$ as defined in \eqref{eq:defconvU}, i.e., 
\begin{equation}\label{eq:objlsc}
    u_n \weak u \text{ in }\UU \quad \Longrightarrow \quad
    \liminf_{n\to\infty} \Psi(u_n) \geq \Psi(u).
\end{equation}
An example for such a functional $\Psi$ will be given in \cref{sec:convmin} below.

\begin{theorem}[Existence of optimal solutions]
\label{thm:existenceOfAGlobalSolution}
	There exists a globally optimal solution of \cref{eq:optprob}.
\end{theorem}

\begin{proof} Based on our above findings, 
	the assertion immediately follows from the standard direct method of calculus of variations.
	Nevertheless, let us shortly sketch the arguments. 
	First, we observe that the triple $(u, \sigma, u_D) \equiv  (u_0, \sigma_0, u_0)$  (constant in time) 
	satisfies the constraints in \eqref{eq:optprob}
    so that the feasible set is nonempty. (At this point, we need the additional regularity $u_0\in \bH^2(\Omega)$.)
    Let $(u_n, \sigma_n, u_{D,n})$ be a minimizing sequence. 
    Then either $(u_0, u_0)$ is already optimal or $J(u_n, u_{D,n}) \leq J(u_0, u_0) < \infty$ for $n\in\N$ sufficiently large.
    Thus, since $\Psi$ is bounded from below, $\{u_{D, n}\}$ is bounded in $H^1(\bH^2(\Omega))$.
    Via continuous and compact embedding, there is thus a subsequence satisfying \eqref{eq:convdiri}. 
    Clearly, the associated limit satisfies the conditions on the initial value in \eqref{eq:optprob}.
    Moreover, according to \cref{prop:contweak}, a subsequence of weak solutions converges weakly in 
    $\UU \times H^1(\Lt^2(\Omega))$ to a weak solution. Thus the weak limit is feasible
    and the weak lower semicontinuity of norms and of $\Psi$ implies its optimality.
\end{proof}

\begin{remark}[More general objectives]
    The proof of existence readily transfers to slightly more general objectives than the one in \eqref{eq:optprob}.
    For instance, one could add a term of the form $\Phi(\sigma)$ with a function $\Phi: H^1(\Lt^2(\Omega))\to \R$, which 
    weakly lower semicontinuous and bounded from below. Since objectives of this form have already been discussed in 
    the companion paper, we restrict ourselves to objectives just depending on $u$ in order to keep the discussion concise. 
    Moreover, one could use other Tikhonov terms different from the $H^1(\bH^2(\Omega))$-norm to ensure 
    the convergence properties in \eqref{eq:convdiri} required for \cref{prop:contweak}.
    For example, thanks to the Aubin-Lions lemma, a Tikhonov term of the form 
    \begin{equation*}
        \frac{\alpha}{2}\, \Big( \|u_D\|_{H^1(\bH^1(\Omega))}^2 + \|u_D\|_{L^2(X)}^2 \Big)
    \end{equation*}
    with any Banach space $X$ embedding compactly in $\bH^1(\Omega)$ (such as e.g.~$\bH^2(\Omega)$)
    is sufficient to guarantee \eqref{eq:convdiri} for (a subsequence of) a minimizing sequence.
    However, in order to shorten presentation, we just consider the $H^1(\bH^2(\Omega))$-norm.
\end{remark}

\section{Yosida Regularization and Reverse Approximation}\label{sec:reverse}
As already mentioned above, the ultimate goal of our analysis is to establish conditions that guarantee 
that optimal solutions to the optimization problem \eqref{eq:optprob} governed by perfect plasticity 
can be approximated via Yosida regularization. 
The most crucial point in this respect is the so-called \emph{reverse approximation}, which 
essentially means to construct a \emph{recovery sequence} for a given perfect plastic solution.
This is a rather challenging task, as \cref{ex:1d} illustrates: one easily verifies that every sequence of 
regularized solutions tends to the linear solution $u(t,x) = 2\,t\,x$ for regularization parameter tending to zero, 
although there are infinitely many other solutions.
There is thus no hope that every perfect plastic 
solution can be approximated via Yosida regularization!
However, when it comes to optimization, there is not only the state (i.e., the solution of the perfect plasticity system), 
but also the \emph{control} variables, which can be used to construct a recovery sequence. 
Unfortunately, the Dirichlet data $u_D$, which serve as control variables in our case, are not sufficient for this purpose.
Instead we need a set of control variables that is rich enough to generate a sufficiently large set of regularized solutions. 
For this purpose, we introduce an \emph{additional control variable in form of distributed loads} and end up with 
the following regularized version of the state equation:
\begin{subequations}\label{eq:regularizedStateEquation}
\begin{alignat}{3}
    -\div 	\sigma_\lambda(t) &= \ell(t) & \quad & \text{in } \bH^{-1}_D(\Omega),\\ 
	\sigma_\lambda(t) &= \mathbb{C} (\symnabla u_\lambda(t) - z_\lambda(t)) & & \text{in }\Lt^2(\Omega),\label{eq:regstress}\\
	\boverdot{z}_\lambda(t) & = \partial I_\lambda(\sigma_\lambda(t)) &&\text{in } \Lt^2(\Omega), \label{eq:regsubdiff}\\
	u_\lambda(t) - u_{D}(t) &\in \bH^1_D(\Omega), \\
	(u_\lambda, \sigma_\lambda)(0) &= (u_0, \sigma_0) & & \text{in } \bH^1(\Omega)\times \Lt^2(\Omega).
\end{alignat}
\end{subequations}
where $\lambda > 0$ is the regularization parameter, $I_\lambda$ is the Yosida regularization of the 
indicator functional, see \eqref{eq:yosida}, and $\ell\in H^1(\bH^{-1}_D(\Omega))$ is the mentioned load.
Existence and uniqueness of a solution to the regularized state equation \cref{eq:regularizedStateEquation} 
follows from Banach's fixed point theorem and can be proven by a reduction of the system to an equation in the variable $z$ only, 
cf.\ e.g.\ \cite[Proposition~3.15]{meywal}. This gives rise to the following

\begin{lemma}[Existence of solutions to the regularized state system, {\cite[Corollary~3.16]{meywal}}]
	For every $\lambda > 0$, $\ell \in H^1(\bH^{-1}_D(\Omega))$, and $u_D \in H^1(\bH^1(\Omega))$
	with $\ell(0) = 0$ and $u_D(0)|_{\Gamma_D} = u_0|_{\Gamma_D}$,
	there exists a unique solution
	$(u_\lambda, \sigma_\lambda, z_\lambda) \in H^1(\bH^1(\Omega)) \times H^1(\Lt^2(\Omega)) \times H^1(\Lt^2(\Omega))$
	of \cref{eq:regularizedStateEquation}. 
	
   The associated solution operator is globally Lipschitz continuous with a Lipschitz constant proportional to $\lambda^{-1}$.
\end{lemma}

The proof of existence is a direct consequence of the Lipschitz continuity of $\partial I_\lambda$ and Banach's 
contraction principle. 
In \cite{meywal}, the external loads are set to zero, but it is straightforward to incorporate them into the 
existence theory. 
The Lipschitz continuity of the solution mapping 
directly follows from the Lipschitz estimate for the Yosida approximation, see e.g.~ \cite[Proposition~55.2(b)]{zeidler3}.

Before we address the approximation properties of this regularization approach 
and its convergence behavior for $\lambda$ tending to zero in \cref{sec:convmin} below, 
see \cref{prop:convYosida}, we first lay the foundations for the construction of a recovery sequence 
in the upcoming three lemmas. 
Unfortunately, as already indicated in the introduction, the passage to the limit in the regularized state equation 
in \cref{prop:convYosida} below requires a rather high regularity of the stress field, and the recovery sequence
has to fulfill this regularity, too, as it is a constraint in the regularized optimal control problem \eqref{eq:optprobYosida}.
The key issue for our reverse approximation argument is therefore to 
improve the regularity of the stress field provided a displacement field with higher regularity is given.
To this end, we first need an auxiliary result on the derivative of the Yosida regularization.
Since the set of admissible stresses admits a pointwise representation by the set $K$, 
the Fr\'echet-derivative of the Yosida regularization does the same, i.e., given an arbitrary $\tau \in \Lt^2(\Omega)$, it holds
\begin{equation}\label{eq:yosidaptwise}
    \partial I_\lambda(\tau)(x) = \frac{1}{\lambda}\big[\tau(x) - \pi_K(\tau(x))\big]
    \quad \text{f.a.a.\ } x\in \Omega,
\end{equation}
where $\pi_K: \Rnns \to \Rnns$ is the projection on $K$.
This pointwise representation allows to derive the following

\begin{lemma}\label{lem:yosida}
    Let $\lambda > 0$, $p>2$, and $\tau \in \Wt^{1,p}(\Omega)$ be arbitrary. 
    Then $\partial I_\lambda(\tau) \in \Wt^{1,p}(\Omega)$  and there holds
    \begin{align}
        \|\partial I_\lambda(\tau)\|_{\Wt^{1,p}(\Omega)} &\leq \frac{1}{\lambda}\, \|\tau\|_{\Wt^{1,p}(\Omega)} \label{eq:yosidaest}
    \intertext{and}
		\big(\partial_i(\partial I_\lambda(\tau)) \colon \partial_i \tau\big)(x) &\geq 0 \quad \text{a.e.\ in } \Omega, 
		\quad \forall\, i = 1, ..., n. \label{eq:weakderivYosi}
    \end{align}
\end{lemma}

\begin{proof}
    As a projection, $\pi_k: \Rnns \to \Rnns$ is globally Lipschitz continuous. 
    Thus, the chain rule for Sobolev functions (see e.g.~\cite[Thm~2.1.11]{ziemer}) implies that 
    $\partial I_\lambda(\tau) \in \Wt^{1,p}(\Omega)$ with 
    \begin{equation}\label{eq:yosidaweak}
        \ddp{}{x_m} [\partial I_\lambda(\tau)]_{ij} = 
        \frac{1}{\lambda}\Big( \ddp{\tau_{ij}}{x_m}  - \sum_{k l} \ddp{}{\tau_{kl}} [\pi_K(\tau)]_{ij} \ddp{\tau_{kl}}{x_m}  \Big) .
    \end{equation}
    Since the Lipschitz constant of the projection equals one, its directional derivative clearly satisfies
    $|\pi_K'(A;B)|_F \leq |B|_F$ for all $A, B\in \Rnns$ and, consequently,     
    \begin{equation*}
    \begin{aligned}
        \partial_m(\partial I_\lambda(\tau)) \colon \partial_m \tau
        & = \frac{1}{\lambda} \big( |\partial_m \tau|_F^2 - \pi_K'(\tau;\partial_m\tau) : \partial_m\tau\big)\geq 0,
    \end{aligned}
    \end{equation*}
    which is \eqref{eq:weakderivYosi}. It is moreover easily seen that $\id - \pi_K: \Rnns \to \Rnns$ is 
    globally Lipschitz with Lipschitz constant 1, too. Thus, for every $A, B\in \Rnns$, there holds $|(\id-\pi_K)'(A;B)|_F \leq |B|_F$.
    Since $(\id - \pi_k)(0) = 0$, the Lipschitz continuity moreover entails $|(\id - \pi_k)(A)|_F \leq |A|_F$ for all $A\in \Rnns$.
    In view of \eqref{eq:yosidaweak}, this yields \eqref{eq:yosidaest}.
\end{proof}

The next lemma addresses the crucial regularity result for the stress field $\sigma_\lambda$ as solution of   
\begin{equation}\label{eq:flowRule_eYosida}
    \e - \mathbb{A}\boverdot{\sigma}_\lambda = \partial I_\lambda(\sigma_\lambda),
    \mediumspace \sigma_\lambda(0) = \sigma_0.
\end{equation}
In the proof of our main result in \cref{thm:main}, an optimal strain rate will play the role of $\e$ and the 
following regularity result will be essential for the construction of a recovery sequence associated with that strain rate.
The required regularity of $w$ will carry over to this optimal strain rate and represents the most restrictive 
assumption of our reverse approximation approach.

\begin{lemma}[Higher regularity of the stress field]\label{lem:w1p}
    Let $\lambda > 0$ be arbitrary and $\e \in L^2(\Lt^2(\Omega)) \cap L^1(\Wt^{1,p}(\Omega))$ with $p\geq 2$ be given. 
    Then \eqref{eq:flowRule_eYosida} 
	 admits a unique solution $\sigma_\lambda \in H^1(\Lt^2(\Omega))\cap L^\infty(\Wt^{1,p}(\Omega))$ and there holds
     \begin{equation}\label{eq:w1pbound}
        \|\sigma_\lambda\|_{L^\infty(\Wt^{1,p}(\Omega))}
        \leq C_p\Big(\|\e\|_{L^1(\Wt^{1,p}(\Omega))} + \|\sigma_0\|_{\Wt^{1,p}(\Omega)}^p\Big)
    \end{equation}	
    with $C_p := p \,\|\Ab\|^{p/2-1}$.
\end{lemma}

\begin{proof}
    \textit{Step~1.~Existence of solutions in $H^1(\Lt^2(\Omega))$:} 
    First we note that \eqref{eq:flowRule_eYosida} is just an ODE in $\Lt^2(\Omega)$ and $\partial I_\lambda$ is globally Lipschitz in $\Lt^2(\Omega)$. 
    Thus, the existence and uniqueness of solutions in $H^1(\Lt^2(\Omega))$ follows from the generalized Picard-Lindel\"of theorem in Banach spaces.
    However, a pointwise projection is in general not Lipschitz continuous in Sobolev spaces. 
    Therefore, we cannot apply this simple argument to show that the solution is an element of $W^{1,1}(\Wt^{1,p}(\Omega))$.

    \textit{Step~2.~Higher regularity in case of smooth data:}
	To prove this, we apply a time discretization scheme, namely the explicit Euler method.
   	At first we consider the case $\e \in C(\Wt^{1,p}(\Omega))$.
	For $N \in \mathbb{N}$ and $n \in \{ 0, ..., N \}$, we set $d_t^N := \frac{T}{N}$ and
	$t_n^N := nd_t^N$ such that $0 = t_0^N < t_1^N < ... < t_N^N = T$.
	Now define $\sigma_0^N := \sigma_0 \in \Wt^{1,p}(\Omega)$ and
	\begin{equation*}
		\sigma^N_n := \sigma^N_{n - 1}
		+ d_t^N\mathbb{C}\big(\e(t_{n-1}^N) - \partial I_\lambda(\sigma^N_{n - 1})\big)
		\in \Wt^{1,p}(\Omega) \quad \text{(by \cref{lem:yosida})}
	\end{equation*}
	such that
	\begin{equation}
	\label{eq:local0}
		\mathbb{A}\frac{\sigma^N_{n} - \sigma^N_{n - 1}}{d_t^N}
		+ \partial I_\lambda(\sigma^N_{n - 1}) = \e(t_{n - 1}^N)
	\end{equation}
	for all $N \in \mathbb{N}$ and $n \in \{ 1, ..., N \}$.
	We define the piecewise linear approximation $\sigma^N \in W^{1,\infty}(\Wt^{1,p}(\Omega))$ by
	\begin{equation*}
		\sigma^N(t) := \sigma_{n-1}^N + \frac{t - t_{n-1}^N}{d_t^N}(\sigma_{n}^N - \sigma_{n-1}^N)
	\end{equation*}
	and the piecewise constant approximation
	$\tilde{\sigma}^N \in L^\infty(\Wt^{1,p}(\Omega))$ by
	$\tilde{\sigma}^N(t) := \sigma_{n-1}^N$ for $t \in [t_{n-1}^N, t_{n}^N)$.
	Using \eqref{eq:yosidaest}, we deduce from \eqref{eq:local0} that 
	\begin{equation*}
		\norm{\sigma^N_n}{\Wt^{1,p}(\Omega)}
		\leq \norm{\sigma_0}{\Wt^{1,p}(\Omega)} 
		+ d_t^N C \Big{(} \sum_{i = 0}^{n - 1} \norm{\sigma^N_{i}}{\Wt^{1,p}(\Omega)} \Big{)}
		+ C \norm{\e}{C(\Wt^{1,p}(\Omega)},
	\end{equation*}
	which, together with the discrete Gronwall lemma (cf.\ \cite[Lemma 5.1 and the following remark]{heywoodRannacher}), shows that 
    $\sigma^N$is bounded in $L^\infty(\Wt^{1,p}(\Omega))$ by a constant	independent of $d_t^N$.
	Thus, again owing to \cref{eq:local0} and \eqref{eq:yosidaest}, 
	$\boverdot{\sigma}^N(t) = \frac{\sigma_n^N - \sigma_{n-1}^N}{d_t^N}$,
	$t \in (t_{n-1}^N, t_{n}^N)$ is also bounded in $L^\infty(\Wt^{1,p}(\Omega))$. Therefore, $\sigma^N$ is bounded in $H^1(\Lt^2(\Omega))$ and 
	consequently, there is a weakly converging subsequence, for simplicity also denoted by $\sigma^N$, such that $\sigma^N \weak \sigma$ in $H^1(\Lt^2(\Omega))$
    and $\sigma^N \weak^* \sigma$ in 	$L^\infty(\Wt^{1,p}(\Omega))$ as $N\to \infty$.  
    Note that, due to the reflexivity of $\Wt^{1,p}(\Omega)$, $L^\infty(\Wt^{1,p}(\Omega))$ can be identified 
    with the dual of $L^1(\Wt^{1,p'}(\Omega))$ so there is a weakly-$\ast$ converging subsequence. 
	It remains to show that $\sigma$ solves \eqref{eq:flowRule_eYosida}. 
    Since $\sigma^N$ is bounded in $W^{1,\infty}(\Wt^{1,p}(\Omega))$ as seen above, 
    we have by compact embeddings that $\sigma^N \to \sigma$ in $C(\Lt^2(\Omega))$.
    Thus, we find for the piecewise constant interpolation that, for every $t\in [t_{n-1}^N, t_n^N)$,
    \begin{equation*}
        \|\tilde\sigma^N(t) - \sigma(t)\|_{\Lt^2(\Omega)}
        \leq \|\sigma^N(t_{n-1}^N) - \sigma(t)\|_\Lt^2(\Omega) \to 0 \quad\text{as } N\to \infty.
    \end{equation*}
	Therefore, \eqref{eq:local0} and the Lipschitz continuity of $\partial I_\lambda$ in $\Lt^2(\Omega)$ give
	\begin{align*}
		& \norm{\mathbb{A} \boverdot{\sigma}^N + \partial I_\lambda (\sigma^N) - \e}{L^2(\Lt^2(\Omega))} \\
		& \leq \norm{\mathbb{A} \boverdot{\sigma}^N
		+ \partial I_\lambda (\tilde{\sigma}^N) - \tilde{\e}^N}{L^2(\Lt^2(\Omega))} 
		+ \norm{\partial I_\lambda (\sigma^N) - \partial I_\lambda (\tilde{\sigma}^N)}{L^2(\Lt^2(\Omega))}
		+ \norm{\tilde{\e}^N - \e}{L^2(\Lt^2(\Omega))} \\
		&\leq \frac{1}{\lambda} \norm{\sigma^N - \tilde{\sigma}^N}{L^2(\Lt^2(\Omega))}
		+ \norm{\tilde{\e}^N - \e}{L^2(\Lt^2(\Omega))} \rightarrow 0 \quad \text{as } N\to \infty,
	\end{align*}
    where $\tilde{\e}^N$ denotes the piecewise constant interpolation of $\e$, which converges strongly in $C(\Lt^2(\Omega))$ to $\e$ 
    thanks to the assumed regularity of $\e$. Therefore, by the weak lower semicontinuity of the $L^2(\Lt^2(\Omega))$-norm, we see 
    that the limit satisfies \eqref{eq:flowRule_eYosida}.	
    
    \textit{Step~3.~Higher regularity for nonsmooth data:}
	Let now $\e \in L^2(\Lt^2(\Omega)) \cap L^1(\Wt^{1,p}(\Omega))$ be arbitrary and take  a sequence $\{\e_n\} \subset C(\Wt^{1,p}(\Omega))$ 
	such that $\e_n \rightarrow \e$ in $L^1(\Wt^{1,p}(\Omega))$.
	Let $\sigma_\lambda \in H^1(\Lt^2(\Omega))$ be the solution of \eqref{eq:flowRule_eYosida} and denote by 
	$\sigma_{\lambda,n} \in H^1(\Lt^2(\Omega)) \cap L^\infty(\Wt^{1,p}(\Omega))$ the solution of
	\begin{align}
	\label{eq:local8}
		\e_n - \mathbb{A}\boverdot{\sigma}_{\lambda, n} = \partial I_\lambda(\sigma_{\lambda, n}),
		\mediumspace \sigma_{\lambda, n}(0) = \sigma_0.
	\end{align}
	Since $\partial I_\lambda : \Lt^2(\Omega)\rightarrow \Lt^2(\Omega)$ is
	monotone, one obtains $\sigma_{\lambda, n} \rightarrow \sigma_\lambda$ in $H^1(\Lt^2(\Omega))$ by standard arguments. 
	Moreover, \eqref{eq:local8} holds almost everywhere in time and space and so that, f.a.a.\ $t\in [0,T]$, 
    \begin{equation*}
        \partial_j \e_n(t) - \Ab\partial_j \boverdot{\sigma}_{\lambda, n}(t) = \partial_j \partial I_\lambda(\sigma_{\lambda, n})(t)
        \quad \text{a.e.\ in } \Omega.
    \end{equation*}
    follows. Testing this equation with 
    $((\Ab \partial_j \sigma_{\lambda, n} : \partial_j \sigma_{\lambda, n})^{p/2 - 1} \partial_j\sigma_{\lambda, n})(t) \in \Wt^{1,p'}(\Omega)$
    and using \eqref{eq:weakderivYosi} leads to
    \begin{equation}\label{eq:w1pest}
    \begin{aligned}
        & \frac{d}{d t}\int_\Omega (\Ab \partial_j \sigma_{\lambda, n} : \partial_j\sigma_{\lambda, n})^{p/2}\d x\\
        & \leq p \int_\Omega (\Ab \partial_j \sigma_{\lambda, n} : \partial_j \sigma_{\lambda, n})^{p/2 - 1} 
        \big(\Ab \partial_j \sigma_{\lambda, n} : \partial_j \boverdot \sigma_{\lambda, n} 
        + \partial_j \partial I_\lambda (\sigma_{\lambda, n}) : \partial_j \sigma_{\lambda, n} \big) \d x\\
        & = p \int_\Omega (\Ab \partial_j \sigma_{\lambda, n} : \partial_j \sigma_{\lambda, n})^{p/2 - 1} 
        \,\partial_j \e_n : \partial\sigma_n\,\d x\\
        & \leq C_p  \, \int_\Omega |\partial_j \e|_F \,|\partial_j\sigma_{\lambda, n}|_F^{p-1}\,\d x
        \leq C_p \, \|\e\|_{\Wt^{1,p}(\Omega)} \, \|\sigma_{\lambda, n}\|_{\Wt^{1,p}(\Omega)}^{p-1},
    \end{aligned}
    \end{equation}
    with $C_p$ as defined in the statement of the lemma.
    Integrating this inequality in time and taking the coercivity of $\Ab$ into account gives
    \begin{equation*}
        \|\sigma_{\lambda,n}\|_{L^\infty(\Wt^{1,p}(\Omega))}
        \leq C_p\Big(\|\e_n\|_{L^1(\Wt^{1,p}(\Omega))} + \|\sigma_0\|_{\Wt^{1,p}(\Omega)}^p\Big).
    \end{equation*}	
	Therefore, $\sigma_{\lambda,n}$ is bounded in $L^\infty(\Wt^{1,p}(\Omega))$ and we can select a weakly-$\ast$ converging 
	subsequence. The uniqueness of the weak limit then gives $\sigma \in L^\infty(\Wt^{1,p}(\Omega))$ as claimed.
	The estimate in \eqref{eq:w1pbound} finally follows from the above inequality and the lower semicontinuity 
	of the norm w.r.t.\ weak-$\ast$ convergence.
\end{proof}

\begin{remark}
    We observe that \eqref{eq:yosidaest}, \eqref{eq:flowRule_eYosida}, and the proven regularity of $\sigma_\lambda$ 
    even imply that $\sigma_\lambda \in W^{1,1}(\Wt^{1,p}(\Omega))$. However, we do not obtain an estimate independent 
    of $\lambda$ in this norm (in contrast to \eqref{eq:w1pbound}) and therefore, this additional regularity is not useful for us.
\end{remark}

\begin{lemma}[{\cite[Section~3]{paperAbstract}}]\label{lem:convYosida}
    Let $\e \in L^2(\Lt^2(\Omega))$ be given and $\lambda \searrow 0$. Then $\sigma_\lambda \rightarrow \sigma$ in $H^1(\Lt^2(\Omega))$,
    where $\sigma$ is the solution of 
    \begin{equation}\label{eq:flowRule_e}
        \e - \mathbb{A}\boverdot{\sigma} \in \partial I_{\KK(\Omega)}(\sigma),	\mediumspace \sigma(0) = \sigma_0.
    \end{equation}    
    Moreover, there holds
    \begin{equation}\label{eq:yosidaestC}
       	\norm{\sigma_\lambda - \sigma}{C(\Lt^2(\Omega))}^2
    	\leq \lambda\,\frac{\norm{\mathbb{C}}{}^2}{\gamma_\mathbb{C}}\,
    	\norm{\e - \mathbb{A}\boverdot{\sigma}}{L^2(\Lt^2(\Omega))}^2,
    \end{equation}
	 where $\gamma_\Cb>0$ is the coercivity constant of $\Cb$.
\end{lemma}

\begin{proof}
    The assertion is proven in \cite{paperAbstract}, but, for convenience of the reader, we shortly sketch the arguments.
    First, observe that $\sigma_\lambda \in H^1(\Lt^2(\Omega))$ and $\sigma\in H^1(\Lt^2(\Omega))$ solve 
    \eqref{eq:flowRule_eYosida} and \eqref{eq:flowRule_e}, respectively, if and only if    
     $z_\lambda := \E - \mathbb{A}\sigma_\lambda$ and $z := \E - \mathbb{A}\sigma$ 
     with $\E(t):= \int_0^t\e(s)\d s$ solve
	\begin{alignat}{3}
		\boverdot{z}_\lambda &= \partial I_\lambda(\mathbb{C}\E - \mathbb{C}z_\lambda),
		&  \quad z_\lambda(0) &= z_0 := -\mathbb{A}\sigma_0	
   \intertext{and}
		\boverdot{z} &\in \partial I_{\KK(\Omega)}(\mathbb{C}\E - \mathbb{C}z),
		& \quad z(0) &= z_0,    
    \end{alignat}
    respectively. These equations are exactly of the form studied in \cite[Section~3]{paperAbstract} with the setting 
    $A := \partial I_{\KK(\Omega)}$, $Q = R := \Cb$, and $\ell := \E$.
	The existence of $\sigma$ in $H^1(\Lt^2(\Omega))$ then follows from \cite[Theorem 3.3]{paperAbstract},
	while the convergence $\sigma_\lambda \rightarrow \sigma$ in $H^1(\Lt^2(\Omega))$ as well as the estimate
	\begin{equation*}
    		\norm{\mathbb{A}(\sigma_\lambda - \sigma)}{C(\Lt^2(\Omega))}^2
    		\leq \frac{\lambda}{\gamma_\mathbb{C}}
    		\norm{\e - \mathbb{A}\boverdot{\sigma}}{L^2(\Lt^2(\Omega))}^2
  	\end{equation*}
	are consequences of \cite[Proposition 3.5]{paperAbstract}.
	(Note that $D(A) = \KK(\Omega)$ is closed and $A^0 \equiv 0$ in this case,
	hence, the assumptions in \cite[Section 2]{paperAbstract} are fulfilled).
	The inequality in \eqref{eq:yosidaestC} now follows easily using
	$\norm{\sigma_\lambda - \sigma}{\Lt^2(\Omega)} = \norm{\mathbb{C}\mathbb{A}(\sigma_\lambda - \sigma)}{\Lt^2(\Omega)}
	\leq \norm{\mathbb{C}}{} \norm{\mathbb{A}(\sigma_\lambda - \sigma)}{\Lt^2(\Omega)}$.
\end{proof}

\begin{remark}
    As a consequence of \eqref{eq:w1pbound}, the solution of \eqref{eq:flowRule_e} 
    is an element of $L^\infty(\Wt^{1,p}(\Omega))$, provided that $\e \in L^1(\Wt^{1,p}(\Omega))$.
    However, we do not need this regularity result for the upcoming analysis.
\end{remark}

As already mentioned, 
in the proof of our final convergence result in \cref{thm:main}, $\symnabla \boverdot{\overline{u}}$ will play the role 
of the function $w$, where $\overline{u}$ is an optimal solution of \eqref{eq:optprob}. 
This already indicates our most restrictive assumption, namely the existence of an optimal solution 
providing the high regularity required for $w$. We will come back to this point in \cref{rem:regassu}.

\section{Convergence of Minimizers}\label{sec:convmin}

We are now in the position to state the regularized optimal control problems. 
Beside the additional control variable $\ell$ required for the reverse approximation, 
they differ from \eqref{eq:optprob} in an additional inequality constraint on the stress field, which 
is needed to improve the regularity of the stress in order to pass to the limit in the regularized state equation, 
see the proof of \cref{prop:convYosida} below. 
This additional regularity of the stresses is unfortunately not enough to pass to the limit in the state system. 
We additionally need to bound the displacement in $\UU$, since this is not guaranteed a priori by the regularized state system
itself, unless the loads fulfill a safe load condition. This however cannot be ensured for the loads arising 
in the construction of the recovery sequence in the proof of our main \cref{thm:main}
(at least, we were not able to verify it). Therefore, we directly enforce this boundedness by a special choice of the 
objective functional as a tracking type objective of the following form:
\begin{equation}\label{eq:objspecial}
    \Psi(u) := \int_0^T \|\symnabla \boverdot u(t) - \mu(t)\|_{\Mf(\Omega;\Rnns)}^2 
	+ \| \boverdot u(t) - v(t)\|_{\bL^1(\Omega)}^2\,\d t
\end{equation}
with a given desired strain rate $\mu\in L^2(\Lt^1(\Omega))$ and a desired displacement rate $v\in L^2(\bL^1(\Omega))$.
Note that this objective trivially fulfills the lower semicontinuity assumption in \eqref{eq:objlsc}.
One could even allow for less regular desired strain rates (in the space of measures), but for convenience, 
we restrict to functions in $L^2(\Lt^1(\Omega))$.
The regularized counterpart of \eqref{eq:optprob} now reads as follows:
\begin{equation}\tag{P\mbox{$_\lambda$}}\label{eq:optprobYosida}
\left\{\;
\begin{aligned}
	\min \quad & J_\lambda(u, u_D, \ell) := \|\symnabla \boverdot u- \mu\|_{L^2(\Lt^1(\Omega))}^2 
	+ \| \boverdot u - v\|_{L^2(\bL^1(\Omega))}^2 \\[-0.5ex]
	& \quad\qquad\quad\quad + \frac{\alpha}{2}\,\|u_D\|_{H^1(\bH^2(\Omega))}^2 
	 + \lambda^{-\theta} \|\ell\|_{L^2(\bH^{-1}_D(\Omega))}^2 + \|\boverdot\ell\|_{L^2(\bH^{-1}_D(\Omega))}^2\\
	\text{s.t.} \quad & u_D \in H^1(\bH^2(\Omega)), \quad \ell\in H^1(\bH^{-1}_D(\Omega)), \\
	& u_D(0)- u_0 \in \bH^1_D(\Omega), \quad \ell(0) = 0,\\
	&(u,\sigma, z) \in \UU \cap L^2(\bH^1(\Omega)) \times L^2(\Lt^2(\Omega)) \times H^1(\Lt^2(\Omega)),\\
	& (u,\sigma, z) \text{ is the solution of 
	\eqref{eq:regularizedStateEquation} w.r.t.\ $u_D$ and $\ell$,}\\
	&\|\boverdot\sigma\|_{L^2(\Lt^2(\Omega))} + \|\sigma\|_{L^s(\Wt^{1,p}(\Omega))} \leq R
\end{aligned}\right.
\end{equation}
with $0 < \theta < 1$ and
\begin{equation}\label{eq:indizes}
    p>n \quad \text{and} \quad  s>\max\Big\{1, \frac{2np}{np + 2(p-n)}\Big\}
\end{equation}
and $R \geq  \|\sigma_0\|_{\Wt^{1,p}(\Omega)}$ to be specified later, see \eqref{eq:Rbound} below.
With the exponents in \eqref{eq:indizes}, \cite[Lemma~4.2(i)]{MMR20} is applicable and tells us that 
$H^1(\Lt^2(\Omega)) \cap L^s(\Wt^{1,p}(\Omega))$ embeds compactly in $L^2(C(\bar\Omega;\Rnns))$, 
which will be useful at several places in the upcoming proofs. The term in the objective associated with $\theta$ will be 
used to force the additional loads to zero in the limit.    
    
\begin{proposition}
    For every $\lambda > 0$, there exists a globally optimal solution of \eqref{eq:optprobYosida}.
\end{proposition}

\begin{proof}
    The proof is almost standard, except for a lack of compactness with regard to the control space.
    Let $(u_n, \sigma_n, z_n, u_{D,n}, \ell_n)$ be a minimizing sequence. As in the proof of \cref{thm:existenceOfAGlobalSolution}, 
    $(u, \sigma, z, u_D, \ell) \equiv (u_0, \sigma_0, \symnabla u_0 - \Ab \sigma_0, u_0, 0)$
    is feasible for \eqref{eq:optprobYosida}. Thus, $\{u_{D,n}, \ell_n\}$ is bounded in 
    $H^1(\bH^2(\Omega)) \times H^1(\bH^{-1}_D(\Omega))$. Hence the 
    Lipschitz continuity of the solution operator associated with \eqref{eq:regularizedStateEquation} implies that 
    $\{(u_n, \sigma_n, z_n)\}$ is bounded in $H^1(\bH^1(\Omega) \times \Lt^2(\Omega) \times \Lt^2(\Omega))$. 
    Therefore, there exist weakly convergent subsequences and we can pass to the limit in \eqref{eq:regularizedStateEquation} 
    except for the nonlinearity in $\partial I_\lambda$. However, the additional constraint on the stress 
    implies that $\sigma_n$ also converges weakly in $H^1(\Lt^2(\Omega)) \cap L^s(\Wt^{1,p}(\Omega))$, which is compactly 
    embedded in $L^2(C(\bar\Omega;\Rnns))$ as mentioned above. Thus $\{\sigma_n\}$ converges strongly in 
    $L^2(C(\bar\Omega;\Rnns))$, which allows to pass to the limit in $\partial I_\lambda(\sigma_n)$ 
    so that the weak limit solves \eqref{eq:regularizedStateEquation}. Moreover, the inequality constraint on $\sigma_n$ 
    is clearly weakly closed so that the weak limit is indeed feasible for \eqref{eq:optprobYosida}. 
    Since the objective is convex and continuous and thus weakly lower semicontinuous, the weak limit is also optimal.
\end{proof}

\begin{proposition}[Convergence of the Yosida regularization with varying loads]\label{prop:convYosida}
    Let $\{\lambda_n\}_{n\in \N}$ be a sequence converging to zero. Suppose moreover that two sequences 
    $\{\ell_n\} \subset H^1(\bH^{-1}_D(\Omega))$ and $\{u_{D,n} \} \subset H^1(\bH^1(\Omega))$ are given and 
    denote the solution of \eqref{eq:regularizedStateEquation} associated with $\lambda_n$, $\ell_n$, and $u_{D,n}$ by $(u_n, \sigma_n, z_n)$.
    Furthermore, we assume that $\{u_{D,n}\}$ satisfies the convergence properties in \eqref{eq:convdiri}, i.e., 
	\begin{equation}\label{eq:convdiri2}
    \begin{gathered}
		u_{D,n} \rightharpoonup u_D \;\text{ in } H^1(\bH^1(\Omega)), \quad
		u_{D,n} \rightarrow u_D \;\text{ in } L^2(\bH^1(\Omega)), \\
		u_{D,n}(T) \rightarrow u_D(T) \;\text{ in } \bH^1(\Omega),    
    \end{gathered}
	\end{equation}
    and that 
    \begin{alignat}{8}
        \ell_n &\weak 0 &\;& \text{ in } L^2(\bH^{-1}_D(\Omega)),  \quad & u_n &\weak u & \; & \text{ in } \UU,\\
        \sigma_n &\weak \sigma &\;& \text{ in } H^1(\Lt^2(\Omega)),  \quad & \sigma_n &\to \sigma 
        &\; & \text{ in } L^2(C(\bar\Omega;\Rnns).\label{eq:convsigma}
    \end{alignat}
    Then $(u, \sigma)$ is a weak solution associated with $u_D$.
\end{proposition}

\begin{proof}
    The arguments are similar to the proof of \cref{prop:contweak}. First, since $\sigma_n \weak \sigma$ in $H^1(\Lt^2(\Omega))$, 
    $\ell_n \weak 0$ in $L^2(\bH^{-1}_D(\Omega))$, and $-\div \sigma_n = \ell_n$ for all $n\in \N$, it follows that $\sigma(t) \in \EE(\Omega)$ 
    f.a.a.\ $t\in (0,T)$. Moreover, from Lemma~3.20 and 3.21 in our compaion paper \cite{meywal}, 
    we deduce that $\sigma(t) \in \KK(\Omega)$ a.e.\ in $(0,T)$, cf.~also the first part of the proof of \cite[Theorem~3.22]{meywal}.
    Moreover, due to $H^1(\bL^{\frac{n}{n-1}}(\Omega)) \embed C(\bL^{\frac{n}{n-1}}(\Omega))$ 
    and $H^1(\Lt^2(\Omega)) \embed C(\Lt^2(\Omega))$, the weak limit satisfies the initial conditions.

    To show the flow rule inequality, let $\tau \in L^2(\Lt^2(\Omega))$ 
    with $\tau(t) \in \Sigma(\Omega) \cap \mathcal{K}(\Omega)$ f.a.a.\ $t \in (0,T)$ be arbitrary. 
    Then, \eqref{eq:regstress} and \eqref{eq:regsubdiff} along with $I_n(a) = 0$ for $a \in \KK(\Omega)$ and $I_n \geq 0$, imply
    \begin{equation*}
	\begin{aligned}
		0 &= \int_0^T I_n(\tau(t)) dt \\
		&\geq	\scalarproduct{\symnabla \boverdot u_n - \mathbb{A}\boverdot\sigma_n}	{\tau - \sigma_n}{L^2(\Lt^2(\Omega))} \\
		&= \scalarproduct{\symnabla \boverdot u_{D,n} - \mathbb{A}\boverdot\sigma_n}	{\tau - \sigma_n}{L^2(\Lt^2(\Omega))}
		- \int_0^T\int_\Omega (\boverdot u_n  - \boverdot u_{D,n})\div \tau\,\d x \d t \\
		&\qquad + \scalarproduct{\symnabla \boverdot u_{D,n} - \symnabla \boverdot u_n}{\sigma_n}{L^2(\Lt^2(\Omega))}.
	\end{aligned}
	\end{equation*}
	Now one can pass to the limit with the first two terms on the right hand side exactly as described at the end of the proof 
	of \cref{prop:contweak}, see \eqref{eq:flowlimit} and \eqref{eq:flowlimit2}. Concerning the last term, we argue as follows:
	Since $\div \sigma = 0$ and $u_n$ satisfies the Dirichlet boundary condition, i.e., $\boverdot u_n = \boverdot u_{D,n}$ 
	on $\Gamma_D$, we obtain
	\begin{equation*}
	\begin{aligned}
        &\big|\scalarproduct{\symnabla \boverdot u_{D,n} - \symnabla \boverdot u_n}{\sigma_n}{L^2(\Lt^2(\Omega))}\big|\\
        &\qquad = \big|\scalarproduct{\symnabla \boverdot u_{D,n} - \symnabla \boverdot u_n}{\sigma_n - \sigma}{L^2(\Lt^2(\Omega))}\big|\\
        &\qquad \leq \|\symnabla \boverdot u_{D,n} - \symnabla \boverdot u_n\|_{L^2(\Lt^1(\Omega))}
        \|\sigma_n - \sigma\|_{L^2(C(\bar\Omega;\Rnns)} \to 0,
	\end{aligned}
	\end{equation*}
	thanks to the boundedness of $u_n$ in $\UU$ and the convergence of $\sigma_n$.
\end{proof}

The last step of the above proof illustrates, where the high regularity of the stress field enforced by the additional inequality 
constraint in \eqref{eq:optprobYosida} comes into play: we need the strong convergence of the stress in 
$L^2(C(\bar\Omega;\Rnns)$ in order to pass to the limit in the flow rule inequality. 
Unfortunately, the recovery sequence needs to be feasible for \eqref{eq:optprobYosida} and thus has to fulfill this 
inequality constraint, too.
Using our results from \cref{sec:reverse}, this  can be guaranteed, provided that there is at least one optimal solution, 
whose strain rate admits higher regularity. This is the most severe restriction for our main result:

\begin{theorem}[Approximation of global minimizers]\label{thm:main}
    Let the objective in \eqref{eq:optprob} be of the form \eqref{eq:objspecial}.
    Assume moreover that there exists a global minimizer
	$(\overline{u}, \overline{\sigma}, \overline{u}_D)$ of \cref{eq:optprob}
	such that $\symnabla\boverdot{\overline{u}} \in L^2(\Lt^2(\Omega)) \cap L^1(\Wt^{1,p}(\Omega))$
	and $\overline{u} - \overline{u}_D\in \bH^1_D(\Omega)$ for all $t\in(0,T)$. 
	Suppose in addition that $R$ in \eqref{eq:optprobYosida} is chosen so large that
	\begin{equation}\label{eq:Rbound}
	    R \geq \frac{1}{\gamma_\Ab}\,\|\overline{u}_D\|_{H^1(\bH^1(\Omega))} + 
	    p \,\|\Ab\|^{p/2-1} \Big(\|\symnabla \boverdot{\overline{ u}}\|_{L^1(\Wt^{1,p}(\Omega))} + \|\sigma_0\|_{\Wt^{1,p}(\Omega)}^p\Big).
	\end{equation}
	Furthermore, let $\{\overline u_\lambda, \overline \sigma_\lambda , \overline z_\lambda, \overline u_{D,\lambda},
	\overline \ell_\lambda \}_{\lambda > 0}$ 	be a sequence of global minimizers of \eqref{eq:optprobYosida} for $\lambda\searrow 0$.
	
	Then there exists an accumulation point of $\{\overline  u_\lambda, \overline \sigma_\lambda , \overline u_{D,\lambda} \}_{\lambda > 0}$
	w.r.t.\ weak convergence in $\UU \times H^1(\Lt^2(\Omega)) \cap L^s(\Wt^{1,p}(\Omega)) \times H^1(\bH^2(\Omega))$.
	Moreover, every such  accumulation point is a global minimizer of \cref{eq:optprob}.

    Furthermore, if $(\tilde u, \tilde\sigma, \tilde u_D)$ is such an accumulation point and 
    $\{\overline  u_\lambda, \overline \sigma_\lambda , \overline u_{D,\lambda} \}_{\lambda > 0}$ the associated sequence converging weakly to it, then
    \begin{alignat}{8}
        \overline{u} &\to \tilde u &\quad & \text{in } H^1(L^{1}(\Omega;\R^n)), 
         \quad & \overline{u}_{D,\lambda} & \to \tilde u_D & \quad & \text{in } H^1(\bH^2(\Omega)), \label{eq:strong1}\\
        \overline{\sigma}_\lambda & \to \tilde \sigma & \quad & \text{in } L^2(C(\bar\Omega;\Rnns)), 
        \quad & \overline{\ell}_\lambda &\to 0 &  \quad &  \text{in } H^1(\bH^{-1}_D(\Omega)). \label{eq:strong2}
    \end{alignat}
\end{theorem}

\begin{proof}
	\emph{Step~1.~Existence of an accumulation point.}
	Since $\{\overline  u_\lambda, \overline \sigma_\lambda \overline{z}_\lambda, \overline u_{D,\lambda}, \overline \ell_\lambda \}_{\lambda > 0}$ is a global
	solution of \eqref{eq:optprobYosida} and the constant tuple 
	$(u, \sigma, z, u_D, \ell) \equiv (u_0, \sigma_0, \symnabla u_0 - \Ab \sigma_0, u_0, 0)$
	is feasible for \eqref{eq:optprobYosida}, we obtain
	\begin{equation}
	\label{eq:local1}
		J_\lambda(\overline u_\lambda, \overline u_{D,\lambda}, \overline \ell_\lambda)
	     \leq J_\lambda(u_0, u_0, 0) =  \frac{\alpha}{2} \norm{u_0}{L^2(\bH^2(\Omega))}^2 =: C < \infty.
	\end{equation}
    Since all $\overline u_\lambda$ share the same initial value and due to the special structure of the 
    objective	 in \eqref{eq:objspecial}, this implies that $\overline u_\lambda$ satisfies the boundedness assumption 
    in \eqref{eq:boundU} such that \cref{lem:convU} yields the existence of a subsequence converging weakly in $\UU$.
    Moreover, the inequality constraint on the stress and the $H^1(H^2)$-norm in the objective immediately yield 
    the boundedness of $\overline{\sigma}_\lambda$ and $\overline{u}_{D,\lambda}$ in their respective spaces, 
    and the reflexivity of the latter imply the existence of a weakly convergent subsequence.
	
	\emph{Step~2.~Feasibility of an accumulation point.}
	Let us now assume that a given subsequence of
	$\{ \overline{u}_\lambda, \overline{\sigma}_\lambda , \overline{u}_{D,\lambda} \}_{\lambda > 0}$,
	denoted by the same symbol for simplicity, converges weakly to $(\tilde{u}, \tilde{\sigma} ,\tilde{u}_{D})$ 
	in $\UU \times H^1(\Lt^2(\Omega)) \cap L^s(\Wt^{1,p}(\Omega)) \times H^1(\bH^2(\Omega))$.
    By the compact embedding of $H^1(\bH^2(\Omega))$ in $C(\bH^1(\Omega))$, this ensures the 
    convergence properties required in \eqref{eq:convdiri2} and in addition $\tilde u_D(0) - u_0 \in \bH^1_D(\Omega)$.
    Moreover, the assumptions on $p$ and $s$ in \eqref{eq:indizes} 
    guarantee that $H^1(\Lt^2(\Omega)) \cap L^s(\Wt^{1,p}(\Omega))$ embeds compactly in $L^2(C(\bar\Omega;\Rnns))$, 
    as already mentioned above, so that \eqref{eq:convsigma} is valid. 
	Furthermore, considering again \cref{eq:local1}, we see
	that $\lambda^{-\theta} \norm{\overline{\ell}_\lambda}{L^2(\bH^{-1}_D(\Omega))}$
	is bounded, hence, $\overline{\ell}_\lambda \rightarrow \ell = 0$ in $L^2(\bH^{-1}_D(\Omega))$ (even with strong convergence).
	Altogether, we observe that the convergence properties in \eqref{eq:convdiri2}--\eqref{eq:convsigma} are fulfilled
	such that \cref{prop:convYosida} yields that the weak accumulation point $(\tilde u, \tilde \sigma)$ is a weak solution 
	associated with $\tilde u_D$ and therefore feasible for the original optimization problem \eqref{eq:optprob}.
	
	\emph{Step ~3.~ Construction of a recovery sequence.}
	First, observe that, since $\overline{u}$ is assumed to be in $H^1(\bH^1(\Omega))$ and to satisfy the Dirichlet boundary conditions, 
    \cref{cor:solutions} gives that 	$(\overline{\sigma}, \overline{u})$ is a strong solution associated with $\overline{u}_D$.

	The recovery sequence for $(\overline{u}, \overline{\sigma}, \overline{u}_D)$ is constructed based on our findings in \cref{sec:reverse}. 
	To be more precise, we apply \cref{lem:w1p} and \cref{lem:convYosida} with $\e = \symnabla \boverdot{\overline u}$.
    According to these lemmas, $\sigma_\lambda \in H^1(\Lt^2(\Omega))$ defined as unique solution of   	
	\begin{equation*}
		\symnabla\boverdot{\overline{u}} - \mathbb{A}\boverdot{\sigma}_\lambda
		= \partial I_\lambda(\sigma_\lambda), \mediumspace \sigma_\lambda(0) = \sigma_0,
	\end{equation*}
	satisfies the bound in \eqref{eq:w1pbound} and converges strongly in $H^1(\Lt^2(\Omega))$ to $\sigma$, which is the solution to 
	\begin{equation*}
		\symnabla\boverdot{\overline{u}} - \mathbb{A}\boverdot{\sigma}
		\in \partial I_{\KK(\Omega)}(\sigma), \mediumspace \sigma(0) = \sigma_0.
	\end{equation*}
	This equation is just the strong form of the flow rule in \eqref{eq:flowrulestrong}.  
	The monotonicity of $\partial I_{\KK(\Omega)}$ immediately gives that \eqref{eq:flowrulestrong} is uniquely solvable.
	Therefore, the limit $\sigma$ coincides with $\overline{\sigma}$, i.e., the stress associated with $\overline{u}_D$.
	If we now define
	\begin{equation*}
	    z_\lambda := \symnabla \overline{u} - \Ab \sigma_\lambda\in H^1(\Lt^2(\Omega)) \quad \text{and} \quad 
	    \ell_\lambda := -\div \sigma_\lambda \in H^1(\bH^{-1}_D(\Omega)),
	\end{equation*}
    then we observe that $(\overline{u}, \sigma_\lambda, z_\lambda)$ is the solution of the regularized plasticity system in 
    \eqref{eq:regularizedStateEquation} w.r.t.\ $\overline{u}_D$ and $\ell_\lambda$. In addition, 
    we have $\ell_\lambda(0) = -\div \sigma_0 = 0$ and 
    $\overline{u}_D(0) - u_0 = \overline{u}_D(0) - \overline{u}(0) \in \bH^1_D(\Omega)$. Therefore, since $\sigma_\lambda$ 
    satisfies the bounds in \eqref{eq:w1pbound} and \eqref{eq:sigmadotbound} (by \cref{lem:sigmabound}), 
    $(\overline{u}, \sigma_\lambda, z_\lambda, \overline{u}_D, \ell_\lambda)$ satisfies all constraints in \eqref{eq:optprobYosida}.
    
    Next we show the convergence of the objective functional.
    As $\overline{\sigma}$ fulfills the equilibrium condition, i.e., $\overline{\sigma} \in \EE(\Omega)$, the convergence of $\sigma_\lambda$ 
    by \cref{lem:convYosida} implies
    \begin{equation*}
        \ell_\lambda  = -\div \sigma_\lambda \to  - \div \overline{\sigma} = 0 \quad \text{in } H^1(\bH^{-1}_D(\Omega)).
    \end{equation*}
    Furthermore, \eqref{eq:yosidaestC} gives
    \begin{equation*}
    \begin{aligned}
        \lambda^{-\theta} \|\ell_\lambda\|_{L^2(\bH^{-1}_D(\Omega))}^2
        &= \lambda^{-\theta} \,\|\div \sigma_\lambda - \div \overline{\sigma}\|_{L^2(\bH^{-1}_D(\Omega))}^2 \\
        &\leq C\, \lambda^{-\theta} \,\|\sigma_\lambda - \overline{\sigma}\|_{L^2(\Lt^2(\Omega))}^2 \\
        &\leq C \,\lambda^{1-\theta}\,	\norm{\symnabla\boverdot{\overline{u}} - \mathbb{A}\boverdot{\overline\sigma}}{L^2(\Lt^2(\Omega))}^2
        \to 0 \quad \text{as }\lambda \searrow 0.
    \end{aligned}
    \end{equation*}
    To summarize, we found that  $(\overline{u}, \sigma_\lambda, z_\lambda, \overline{u}_D, \ell_\lambda)$ is feasible 
    for \eqref{eq:optprobYosida} and fulfills 
	\begin{equation}
		J_\lambda(\overline{u}, \overline{u}_D, \ell_\lambda) \rightarrow J(\overline{u}, \overline{u}_D).	
	\end{equation}
	
	\emph{Step~4.~Strong convergence and global minimizer.}
	The feasibility and the convergence of the recovery sequence and the optimality of $(\overline{u}_\lambda, \overline{u}_{D,\lambda}, 
	\overline{\ell}_\lambda)$ give
	\begin{equation}\label{eq:local5}
    \begin{aligned}
		J(\tilde{u}, \tilde{u}_D)
		&\leq \liminf_{\lambda \searrow 0} J(\overline{u}_\lambda, \overline{u}_{D,\lambda}) \\
		& \leq \limsup_{\lambda \searrow 0} J(\overline{u}_\lambda, \overline{u}_{D,\lambda}) \\
		&\leq \limsup_{\lambda \searrow 0} J_\lambda(\overline{u}_\lambda, \overline{u}_{D,\lambda}, \overline{\ell}_\lambda)
		\leq \limsup_{\lambda \searrow 0} J_\lambda(\overline{u}, \overline{u}_{D}, \ell_\lambda) 
		= J(\overline{u}, \overline{u}_D),
    \end{aligned}
    \end{equation}	
    which, together with the feasibility of $(\tilde{u}, \tilde{\sigma}, \tilde{u}_D)$ for \eqref{eq:optprob} shown in step~2,
    implies that $(\tilde{u}, \tilde{\sigma}, \tilde{u}_D)$ is a global minimizer of \cref{eq:optprob}.
    
    To show the strong convergence in \eqref{eq:strong1} and \eqref{eq:strong2}, we first observe that 
    \eqref{eq:local5} yields $J(\overline{u}_\lambda, \overline{u}_{D,\lambda}) \rightarrow J(\tilde{u}, \tilde{u}_D)$, 
	from which we deduce the convergence of the norms $\|\boverdot{\overline u}_\lambda\|_{L^2(\bL^1(\Omega))}$ 
	and $\|\overline{u}_{D,\lambda}\|_{H^1(\bH^2(\Omega))}$ to 
	$\|\boverdot{\tilde u}\|_{L^2(\bL^1(\Omega))}$ and $\|\tilde u_{D}\|_{H^1(\bH^2(\Omega))}$, respectively. 
	Since both norms are Kadec norms and we already have weak convergence in the respective spaces, 
	this implies \eqref{eq:strong1}.
    Similarly, \eqref{eq:local5} yields $\|\overline{\ell}_\lambda\|_{H^1(\bH^{-1}_D(\Omega))} \to 0$. 
    Finally, the strong convergence of the stresses follows from the compact embedding of 
    $H^1(\Lt^2(\Omega)) \cap L^s(\Wt^{1,p}(\Omega))$ in $L^2(C(\bar\Omega;\Rnns))$, already used above.
\end{proof}

Some comments concerning our approximation result are in order:

\begin{remark}[Crucial regularity assumption]\label{rem:regassu}
        The assumption of existence of a global minimizer $(\overline{u}, \overline{\sigma}, \overline{u}_D)$ 
        with the properties listed in \cref{thm:main} is admittedly very restrictive. Notice in particular that the regularity assumptions 
        on $\overline{u}$ imply that $(\overline{u}, \overline{\sigma})$ is a \emph{strong solution} w.r.t.\ $\overline{u}_D$, 
        whose existence can in general not be guaranteed. 
        The regularity assumption however seems to be indispensable, as the above proof demonstrates:
        In order to pass to the limit in the flow rule inequality to show the feasibility of an accumulation point in step~2 of the proof, 
        we need the additional regularity of the stress ensured by the inequality constraint in \eqref{eq:optprobYosida}. 
        The generic regularity of the stress, which is $H^1(\Lt^2(\Omega))$ (see \cref{lem:sigmabound}), 
        is by far not sufficient for this passage to the limit. It therefore appears to be unavoidable to enforce the required regularity 
        by additional inequality constraints in \eqref{eq:optprobYosida}. The elements of the  recovery sequence however have to be feasible
        for \eqref{eq:optprobYosida} and thus have to fulfill this inequality constraint, too. 
        As the generic regularity of the stress is $H^1(\Lt^2(\Omega))$, it is not possible to guarantee this constraint to be fulfilled 
        without further hypotheses on the recovery sequence and its limit, respectively. At the end, this leads to the regularity assumption 
        on $\symnabla \boverdot{\overline{u}}$.

        We however emphasize that we do not require the existence of a strong 
        solution with the addition regularity of the strain rate
        for every Dirichlet displacement $u_D \in H^1(\bH^2(\Omega))$ (which would really be unrealistic),
        but only for one optimal $\overline{u}_D$. 
        (Of course, there might be many optimal solutions, since \eqref{eq:optprob} is a non-convex problem).
        Whether an optimal solution fulfilling these regularity assumptions exists or not, clearly depends on the data, 
        especially on the smoothness of the desired strain rate $\mu$ in \eqref{eq:objspecial}.
\end{remark}

\begin{remark}[Extensions and modifications of the approximation result]\label{rem:blablablub}
    \begin{enumerate}
        \item[(i)] One essential drawback of the approximation result is that the bound $R$ given in \eqref{eq:Rbound} 
        depends on the unknown solution $(\overline{u}, \overline{u}_D)$ and is therefore in general unknown, too. 
        One could replace the inequality constraints on the stress involving this bound in \eqref{eq:optprobYosida} 
        by an additional tracking term in the objective of the form 
        $\|\sigma - \sigma_d\|_{H^1(\Lt^2(\Omega))}^2 + \|\sigma - \sigma_d\|_X$ with a given 
        desired stress distribution $\sigma_d$ and a reflexive Banach space $X$ with the following properties:
        On the one hand, $H^1(\Lt^2(\Omega)) \cap X$ should compactly embed in $L^2(C(\bar\Omega;\Rnns))$.
        On the other hand, $H^1(\Lt^2(\Omega)) \cap L^\infty(\Wt^{1,p}(\Omega))$ should compactly be embedded in $X$. 
        Provided these embeddings hold, the steps 2 and 3 of the previous proof can easily be adapted. 
        At this point, one benefits from the strong convergence of the recovery sequence in $H^1(\Lt^2(\Omega))$ by \cref{lem:convYosida}.
        \item[(ii)] The above analysis is restricted to objectives of the type \eqref{eq:objspecial} or other types
        of objectives ensuring the boundedness of $\{\overline{u}_\lambda\}$ in $\UU$. This bound cannot be deduced 
        from the regularized plasticity system in \eqref{eq:regularizedStateEquation} unless the loads fulfill a 
        \emph{safe load condition}, see \cite{suquet}. One could thus allow for more general objectives, if 
        a safe load condition would be included in the set of constraints in \eqref{eq:optprobYosida}. 
        We were however not able to find a safe load condition that is satisfied by the loads associated with the recovery sequence.
        This is due to several reasons, among these a lack of regularity of the recovery sequence. This issue is subject to future research.
        \item[(iii)] By contrast, it is well possible to consider objectives, which give the boundedness of 
        the displacement in more regular spaces such as $H^1(\bH^1(\Omega))$. In this case, the inequality constraints on the stress in 
        \eqref{eq:optprobYosida} can be weakened or even be completely left out, since the higher regularity of the displacement 
        enables the passage to the limit at the end of the proof of \cref{lem:convYosida}. Such a setting is 
        treated in \cite{walther}.
        \item[(iv)] 	We have chosen the space $H^1(\bH^2(\Omega))$ as the 
    	control space for the Dirichlet displacement in order to guarantee the compact embeddings in step~2 of the above proof 
    	and in the proof of \cref{thm:existenceOfAGlobalSolution}. 
    	Of course, one might want to avoid the $\bH^2(\Omega)$-norm in the objective, which 
    	could be achieved by an additional (pseudo-)force-to-Dirichlet-map, for example by solving an additional
    	linear elasticity system. This strategy was employed in \cite[Subsection~6.1]{meywal}.
    \end{enumerate}
\end{remark}

\begin{remark}[Numerical treatment of \eqref{eq:optprobYosida}]
    Although they are still nonsmooth optimization problems, 
    the regularized problems in \eqref{eq:optprobYosida} offer ample possibilities for a numerical treatment. 
    A popular strategy is to further regularize the problem by smoothing the Yosida approximation $\partial I_\lambda$. 
    This has been used for the numerical computations in the companion paper \cite{meywal}. Moreover, 
    the non-smooth objective in \eqref{eq:optprobYosida} calls for an additional regularization of the $L^1$-norms 
    for instance in terms of a Huber-regularization. In this way, one ends up with a smooth optimal control problem, which 
    can be treated by the classical adjoint approach. Our convergence result in \cref{thm:main} implies that, 
    under the certainly restrictive assumptions of this theorem, there is an 
    optimal solution of the original optimization problem governed by the perfect plasticity system that 
    can be approximated by this procedure.
\end{remark}

\section*{Acknowledgments}
We are very grateful to Hannes Meinlschmidt (RICAM Linz) for various fruitful discussions on regularity issues 
associated with the construction of the recovery sequence.

\begin{appendix}

\section{Auxiliary results}

\begin{lemma}\label{lem:density}
    Let $\Omega\subset \R^n$ be a bounded domain and
    $X$ be a Banach space and $A\subset X$ be a convex and closed set with $0 \in A^\circ$. 
    Set $\AA (\Omega):= \{v\in L^2(\Omega;X) : v\in A \text{ a.e.~in }\Omega\}$. Then 
    $C^\infty_c(\Omega;X)\cap \AA(\Omega)$ is dense in $\AA(\Omega)$.
\end{lemma}

\begin{proof}
    Let $v \in \AA(\Omega)$ and $\varepsilon \in (0, 1)$ be arbitrary. By assumption there exists $\delta>0$ such that 
    $\overline{B_X(0,\delta)} \subset A$.  We set $\overline{v} := (1 - \varepsilon) v$ 
	and select a sequence $\sequence{v}{n} \subset C^\infty_c(\Omega; X)$ such that
    \begin{equation}\label{eq:taudense}
	    \norm{v_n - \overline{v}}{\lz{X}}^2 \leq \frac{\delta^2 \varepsilon^3}{4n}
	    \quad \forall\, n\in \N.
    \end{equation}    	
    We moreover define
	\begin{align*}
		S_n^c := \{ x \in \Omega : v_n(x) \in X \setminus A^\circ \}, \quad
		S_n^o := \{ x \in \Omega : v_n(x) \in X \setminus (1 - \tfrac{\varepsilon}{2})A \}.
	\end{align*}
	Hence, $S_n^c \subset S_n^o$ and, by continuity  and compact support of $v_n$, 
	$S_n^c$ is compact, while $S_n^o$ is open. Thus, for every $n\in \N$, 
    there is a function $\varphi_n \in C^\infty(\Rn; [0,1])$ with $\varphi_n \equiv 1$ in $\Rn \setminus S_n^o$ and
	$\varphi_n \equiv 0$ in $S_n^c$. Furthermore, if 
	$\| v_n(x) - \overline{v}(x) \|_X \leq \frac{\varepsilon}{2}\,\delta$, then the convexity of $A$ and 
	$\overline{B_X(0,\delta)}\subset A$ imply
	\begin{align*}
		\frac{v_n(x)}{1 - \frac{\varepsilon}{2}} = 
		\frac{1 - \varepsilon}{1 - \frac{\varepsilon}{2}}\,v(x) 
		+ \Big(1 - \frac{1 - \varepsilon}{1 - \frac{\varepsilon}{2}}\Big)\,
		\frac{2}{\varepsilon}\,\big(v_n(x) - \overline{v}(x)\big) \in A
	\end{align*}
	Therefore, we obtain by contraposition that
	\begin{align*}
		\norm{v_n - \overline{v}}{\lz{X}}^2 \geq \int_{S_n^o} \| v_n - \overline{v} \|_X^2\d x
		\geq \frac{\varepsilon^2}{4} \,\delta^2 \,| S_n^o |
	\end{align*}
	so that \eqref{eq:taudense} yields $| S_n^c | \leq | S_n^o | \leq \varepsilon/n$. Thus, 
	due to Lebesgue's dominated convergence theorem, there exists
	$N = N(\varepsilon) \in \mathbb{N}$ such that
	\begin{equation}\label{eq:tauSo}
		\norm{\overline{v}}{L^2(S_N^o; X)} \leq \norm{v}{L^2(S_N^o; X)} \leq \varepsilon.
	\end{equation}
    Now we define $v_s := \varphi_N\, v_N$. Then, by construction 
    $v_s \in \AA(\Omega) \cap C^\infty_c(\Omega; X)$ and, in addition, 
    \eqref{eq:taudense} and \eqref{eq:tauSo} imply
	\begin{align*}
		\norm{v - v_s}{\lz{X}} &\leq
		\norm{v - \overline{v}}{\lz{X}} + \norm{\overline{v} - v_N}{\lz{X}}
		+ \norm{v_N - v_s}{\lz{X}} \\
		&\leq \varepsilon \norm{v}{\lz{X}} + \norm{\overline{v} - v_N}{\lz{X}}
		+ \norm{v_N}{L^2(S_N^o; X)} \\
		&\leq \varepsilon \norm{v}{\lz{X}} + 2\,\norm{\overline{v} - v_N}{\lz{X}}
		+ \norm{\overline{v}}{L^2(S_N^o; X)} \\
		&\leq \varepsilon \Big(\norm{v}{\lz{X}} + \frac{\delta\sqrt{\varepsilon}}{\sqrt{N}} + 1\Big).
	\end{align*}
	Since $\varepsilon$ was arbitrary, this finishes the proof.
\end{proof}

\end{appendix}

\bibliographystyle{siamplain}
\bibliography{perfectplastdisp}
\end{document}